\documentclass[12pt, twoside]{amsart}

\usepackage{enumerate} 
\usepackage{enumitem}
\usepackage{hyperref}

\usepackage{latexsym}
\usepackage{pxfonts}
\usepackage{color}

\usepackage{graphicx} 

\usepackage{amssymb}
\usepackage{amsthm}
\usepackage{amscd}
\usepackage{amssymb} 
\usepackage{mathrsfs}          
\usepackage[all]{xy}
\usepackage {changes}

\usepackage[a4paper, total={14.5cm, 23cm}]{geometry}
\calclayout

\newcommand{\NS}[0]{{\operatorname{NS}}}

\newcommand{\minimo}{\min (\dim X+\sigma(M_1 \dots, M_k), 2\dim X)}

  \newtheorem{theorem}{Theorem}[section]
  \newtheorem{lemma}[theorem]{Lemma}
  \newtheorem{proposition}[theorem]{Proposition}
  \newtheorem{corollary}[theorem]{Corollary}
  
  \newtheorem{conjecture}[theorem]{Conjecture}

\theoremstyle{definition}
  \newtheorem{notation}[theorem]{Notation}
  \newtheorem{definition}[theorem]{Definition}

  \newtheorem{question}[theorem]{Question}

\newtheorem{remark}[theorem]{Remark}

\theoremstyle{remark}

\numberwithin{equation}{section}

\title[Characterization of products of projective spaces]{Characterization of products of projective spaces \\ via nef complexity}

\author[Enwright]{Joshua Enwright}
\address{UCLA Mathematics Department, Box 951555, Los Angeles, CA 90095-1555, USA
}
\email{jlenwright1@math.ucla.edu }

\author[Filipazzi]{Stefano Filipazzi}
\address{Department of Mathematics, Duke University, 120 Science Drive, 117 Physics Building, Campus Box 90320, Durham, NC 27708-0320, USA}
\email{stefano.filipazzi@duke.edu}

\author[Gongyo]{Yoshinori Gongyo}
\address{Graduate School of Mathematical Sciences, the University of Tokyo, 3-8-1 Komaba, Meguro-ku, Tokyo 153-8914, Japan}
\email{gongyo@ms.u-tokyo.ac.jp}

\author[Moraga]{Joaqu\'in Moraga}
\address{UCLA Mathematics Department, Box 951555, Los Angeles, CA 90095-1555, USA}
\email{jmoraga@math.ucla.edu}

\author[Svaldi]{Roberto Svaldi}
\address{Dipartimento di Matematica "F. Enriques", Via Saldini 50, 20133 Milano (MI), Italy}
\email{roberto.svaldi@unimi.it}

\author[Wang]{Chengxi Wang}
\address{Yau Mathematical Sciences Center, Jingzhai, Tsinghua University, Haidian District, Beijing, China 100084}
\email{chxwang@tsinghua.edu.cn}

\author[Watanabe]{Kiwamu Watanabe}
\address{Department of Mathematics, Faculty of Science and Engineering, Chuo University. 1-13-27 Kasuga, Bunkyo-ku, Tokyo 112-8551, Japan}
\email{watanabe@math.chuo-u.ac.jp}


\begin{document}

\begin{abstract}
We define the {\em nef complexity} of a projective variety $X$. This invariant compares $\dim X+\rho(X)$
with the sum of the coefficients of nef partitions of $-K_X$. We prove that the nef complexity is non-negative
and zero precisely for products of projective spaces.
We classify smooth Fano threefolds with nef complexity at most one. In a similar vein, we prove Mukai's conjecture for smooth Fano varieties for which every extremal contraction is of fiber type and study smooth images of products of projective spaces.
Along the way, we answer positively a question of J. Starr regarding the nef cone of smooth Fano varieties.
\end{abstract}

\maketitle

\setcounter{tocdepth}{1} 
\tableofcontents

\section{Introduction}
Throughout this article we work over the field 
$\mathbb C$
of complex numbers.

Fano varieties constitute one of the three fundamental building blocks in the classification of algebraic varieties. 
The simplest examples of Fano varieties are given by projective spaces, which are fundamental and ubiquitous varieties in projective algebraic geometry. 
Such ubiquity has inspired and motivated the search for simple characterizations of projective spaces among smooth Fano and even algebraic varieties at large. 

This natural problem has seen many interesting and profound developments over the years.
To name some of the most well-known: Mori's solution to Hartshorne's conjecture~\cite{Mor79};
or, the cohomological and topological characterization provided by the Kobayashi--Ochiai Theorem~\cite{kooc}.
While the former result characterizes projective spaces as the only smooth projective varieties with ample tangent bundle, 
the latter shows that the index of the canonical class within the Picard group of a Fano variety of dimension $n$
is bounded from above by $n+1$, and that equality is achieved only by $\mathbb P^n$.

Several other characterizations of projective spaces are known from the viewpoint of
deformations of rational curves~\cite{CMSB97}, 
of topology~\cite{HK57}, 
of the existence of special K\"ahler--Einstein metrics~\cite{sY}, 
cohomology~\cite{ADK08}, etc.

\medskip

Products of projective spaces can be considered as the next example of simple smooth projective varieties, after projective spaces.
Unlike the case of projective spaces, 
not many characterizations are known for products of projective spaces.
In this direction, a motivating question is offered by the following very famous conjecture of Mukai \cite{mukai}.

\begin{conjecture}
\label{mukai conjecture}
Let $X$ be a smooth Fano variety.
Let $\rho(X)$ be the Picard number of $X$, and let
\begin{align*}
    i(X) \coloneqq \max\{r \in \mathbb{ Z}\mid  -K_X \sim_{\mathbb{Z}} r H \text{ for some Cartier divisor $H$} \}
\end{align*}
be its Fano index.
Then, 
\begin{align}
\label{diseq:mukai.conj}
\dim X+\rho(X)-i(X) \cdot \rho(X) \geq 0.
\end{align}
Moreover, equality holds if and only if $X \simeq \mathbb{P}^{i(X)-1} \times \cdots  \times \mathbb{P}^{i(X)-1}$.
\end{conjecture}

The literature on Mukai's conjecture is rather vast, see, for example,~\cite{And09,bcdd,Fuj14b,Fuj14,Fuj16,Rei24,Suz20,w-mukai}. 
A fruitful approach that has been heavily investigated over the past few decades was to study the above conjecture using families of deformation of rational curves on $X$, cf., for example,~\cite{bcdd,w-mukai}. 

Inspired by this technique, several mathematicians have also considered the {\em generalized Mukai conjecture} 
which, for a smooth Fano $X$, predicates the 
same inequality as in
\eqref{diseq:mukai.conj}
with the Fano index 
$i(X)$ 
replaced by the {\em pseudo-index} $\iota(X)$, 
see, e.g.,~\cite{And09,aco,bcdd,FUJ19,GH17,Suz20,w-mukai};
the pseudo-index of a smooth Fano variety $X$ is defined by
$    \iota(X) 
    \coloneqq
    \min
    \left\{
    -K_X \cdot C
    \; \middle \vert \;
    \text{$C \subset X$ is a proper rational curve}
    \right\}
$.

\medskip

In~\cite{g-mukai}, in an attempt to connect the characterization of projective spaces and decompositions of $-K_X$ as a sum of nef line bundles, the third author proposed the following definition.

 \begin{definition}\label{total-index}
 Let $X$ be a normal projective variety.
 Assume that $-K_X$ is $\mathbb Q$-Cartier, nef, and $K_X \not \equiv 0$.
 The {\em total index}
 $\tau_X$
 of 
 $X$ 
 is defined as 
 \[ 
 \tau_X 
 \coloneqq
 \sup 
 \left \{
 \sum_{i=1}^k a_i 
 \; \middle \vert \;
 \begin{array}{l}
 \text{for $i=1, \dots, k$, } 
 a_i \in \mathbb{Q}_{>0}, 
 \text{ and }
 \exists \
 \text{$L_i$ nef and Cartier on $X$}
    \\
 \text{such that }
 L_i \not \equiv 0
 \text{ and }
 -K_X
 \equiv 
 \sum_{i=1}^k a_i L_i
 \end{array}
 \right \}.
 \]
 The {\em nef complexity}
 $c_X$
 of 
 $X$ 
 is 
 defined by
 $c_X \coloneqq \dim X + \rho(X) - \tau_X$.
 \end{definition} 

The nef complexity is a variation on the theme of finding quantities that measure (either locally at a germ, or globally, in the proper case) how to decompose the anticanonical class using nef (resp. effective divisors) under suitable assumptions on the singularities of the underlying variety (resp. log pair).
Other notions of complexity have been previously defined
\cite{shok.complements} 
and used to characterize toric varieties~\cite{BMSZ} and toric singularities~\cite{MS21}.

Let us note that if $X$ is a Mori dream space, which is guaranteed in the $\mathbb Q$-factorial klt Fano case by \cite{bchm}, 
then $\tau_X$ is indeed a minimum as the nef cone of $X$ is finitely generated. 

In \cite{g-mukai}, the third author also proposed the following conjecture:

 \begin{conjecture}[Mukai-type conjecture]
 \label{Mukai type conjecture}
 Let $X$ be a smooth Fano variety.
 Then $c_X\geq 0$. 
 If the equality holds, then $X$ is a product of projective spaces.
 \end{conjecture}

 We remark that Conjectures \ref{mukai conjecture} and  \ref{Mukai type conjecture} are equivalent to the Kobayashi--Ochiai theorem \cite{kooc} when the Picard number is 1, and thus do hold under such hypothesis.

\medskip
 
One of our main results in this article is a proof of the following more general version of Conjecture ~\ref{Mukai type conjecture}, which deals also with singular Fano varieties, and beyond, using generalized pairs.
We refer the reader to~\cite{Bir20} for the language of generalized pairs, see also Section~\ref{preparation}.

\begin{theorem}[{cf.~Theorem \ref{gen.mukai.thm}}]\label{mukai type gklt}
Let $X$ be a
normal projective variety.
Assume that there exists an effective divisor 
$\Delta$
and a sum of nef divisors
$\sum_{i=1}^k m_iM_i$
on 
$X$
such that 
\begin{align*}
K_X+
\Delta+ 
\sum_{i=1}^k m_iM_i
\equiv 0
\end{align*}
and
\begin{enumerate}
\item 
$(X, \Delta+\sum_{i=1}^k m_iM_i)$ is generalized klt;
\item 
for all 
$i=1, \dots, k$, 
$M_i$ 
is nef and Cartier,
and $M_i \not\equiv 0$;
and 
\item 
$\Delta+\sum_{i=1}^k m_iM_i$
is big.
\end{enumerate}
Then 
\begin{align}
\label{gen.mukai.thm.eqn-intro}
\sum_{i} m_i \leq 
\dim X+\rho(X)
\end{align}
Moreover, if the equality holds in \eqref{gen.mukai.thm.eqn-intro},
then 
$X$ is isomorphic to a product of projective spaces, i.e.,
$X\simeq \prod_{j=1}^l \mathbb P^{t_j}$.
\end{theorem}

In \cite{g-mukai}, the third author proved that the Mukai-type conjecture is implied by the Kawamata--Ambro effective non-vanishing conjecture (see~\cite{Amb99b}).
Moreover, the third and fourth authors proved Conjecture \ref{Mukai type conjecture} in dimension $2$ in \cite{gm}.
In this paper, we prove Conjecture \ref{Mukai type conjecture}. This leads to a divisorial characterization of products of projective spaces. 

As a corollary, we deduce that Conjecture~\ref{Mukai type conjecture} is also valid for klt Fano varieties. 
More precisely, we obtain:

\begin{corollary}\label{cor:mukai-type-klt-fano}
Let $X$ be a klt Fano variety. 
Then $c_X\geq 0$. 
Furthermore, if $c_X=0$, then $X$ is isomorphic to a product of projective spaces.
\end{corollary}

Occhetta \cite{occ} proved a characterization of products of projective spaces using unsplit covering families of rational curves.
Building on Occhetta's result, Araujo 
\cite[Theorem~1.3]{Ara06} obtained another characterization based on the geometry of minimal rational curves and their tangent directions. Furthermore, Fujita 
\cite[Theorem~2.11]{Fuji14c} obtained a criterion for varieties being isomorphic to the products of projective spaces in terms of length of extremal rays via Araujo's result.
In contrast, Corollary~\ref{cor:mukai-type-klt-fano} of this paper provides a new characterization formulated purely in terms of a numerical decomposition of the anticanonical divisor. In particular, Fujita's theorem 
\cite[Theorem~1.3]{Fuj22} 
on Fano manifolds carrying many free divisors appears as a special case of our corollary.

Theorem~\ref{mukai type gklt} provides the upper bound
$\dim X + \rho(X)$
on the total index for generalized klt log Calabi--Yau pairs. 
However, it is also possible to find an upper bound which only depends on the dimension of the underlying variety.

\begin{theorem}
\label{introthm:charct-p1n}
Let $X$ be a klt Fano variety.
Assume that there exists a decomposition 
$-K_X\equiv \sum_{i=1}^k b_i M_i$ 
such that for all 
$i=1, \dots, k$, 
$b_i>0$,
and 
$M_i$ is a nef Cartier divisor on 
$X$, 
$M_i \not \equiv 0$.
Then, 
\[
\sum_{i=1}^k b_i\leq 2\dim X.\] 
Moreover, if equality holds, then
$X\simeq (\mathbb{P}^1)^{\dim X}$.
\end{theorem}

In the proof of Theorems~\ref{mukai type gklt} and~\ref{introthm:charct-p1n}, an important step is represented by the following Fujita-type theorem that provides an upper bound for the pseudo-effective threshold of the canonical class with respect to nef and big Cartier divisors.

\begin{theorem}[{cf.~Theorem \ref{numerical Fujita type}}]\label{intro-fujita type-1}
Let $X$ be an $n$-dimensional normal projective variety and $L_1, \cdots, L_{n+1}$ be nef and big Cartier divisors on $X$. 
Then $K_X +L_1+\cdots+L_{n+1}$ is pseudo-effective. 
\end{theorem}

The above theorem follows immediately from Mori's cone theorem and the estimate in the length of extremal rays when all 
$L_i$
are assumed to be ample and 
$X$ 
is log canonical, cf. \cite{length.extr.ray}:
indeed, under such assumption,
$K_X +L_1+\cdots+L_{n+1}$ 
is even nef rather than pseudo-effective. 

By combining Theorem \ref{intro-fujita type-1} with the Kobayashi--Ochiai theorem, we obtain the following characterization of projective space:

\begin{theorem}[{cf.~Theorem \ref{Fano-case}}]\label{intro-fujita type}
Let $X$ be a klt Fano variety of dimension $n$.
Assume that there exists a decomposition 
$-K_X=\sum_{i=1}^k a_i M_i$ 
such that for all 
$i=1, \dots, k$, 
$a_i \in \mathbb{R}_{>0}$,
and $M_i$ is a nef and big Cartier divisor.
Then
$\sum_i a_i\leq n+1$, 
and if equality holds, then 
$X$ is isomorphic to the projective space $\mathbb P ^n$ and 
$\mathcal O_X(M_i)$ 
is isomorphic to 
$\mathcal O_{\mathbb P^n}(1)$.
\end{theorem}

The above theorem is Theorem \ref{mukai type gklt} in the case when every $M_i$ is big.
So, once Theorem~\ref{intro-fujita type} is established,
we can assume that some $M_i$ is not big in the statement of Theorem~\ref{mukai type gklt}, thus inducing a fibration for the variety $X$. 
This gives a natural set up to proceed by induction on the dimension. The details of the inductive argument
for the proof of Theorem~\ref{mukai type gklt}
are given in Section~\ref{induction:proof}.

\medskip

Once we understand the varieties of nef complexity zero, it is natural to try to classify varieties of
nef complexity $<1$.
We expect that $c_X<1$ implies $c_X=0$, however, 
our proof does not give this statement.
With our current techniques, we can prove this
statement for smooth Fano threefolds, and we further classify
smooth Fano threefolds of nef complexity one.

\begin{theorem}
[{cf. Theorem~\ref{thm:smooth-fano-3fold-cx=1}}]
\label{introthm:3-fold-case}
Let $X$ be a smooth Fano threefold. 
If $c_X<1$, then $c_X=0$, and $X$ is either
$\mathbb{P}^3, \mathbb{P}^2\times\mathbb{P}^1,$ or 
$(\mathbb{P}^1)^3$. 
If $c_X=1$, then $X$ is isomorphic to one of the following:
\begin{itemize}
    \item a smooth quadric threefold ${Q}^3$;
    \item the projectivization of the tangent bundle of $\mathbb{P}^2$; 
    \item the blow-up of $\mathbb{P}^3$ at a point;
    \item the blow-up of $\mathbb{P}^3$ along a line; or 
    \item the product $\mathbb{P}^1\times {\rm Bl}_p(\mathbb{P}^2)$.
\end{itemize}
\end{theorem}

Theorem~\ref{mukai type gklt} and Theorem~\ref{introthm:charct-p1n} are both related to the combinatorics of reflexive polytopes. 
Indeed, the total index of a toric variety is related to nef partitions of polytopes (see, e.g.,~\cite{Bat94,Li16}). 
A nef partition of a polytope $P$ is a decomposition such as the Minkowski sum of integral polytopes (see, e.g.,~\cite[Definition 2]{ANSW09}).
In~\cite{Bat08}, Batyrev proved the following proposition: 

\begin{proposition}\label{prop:cube}
Let $P \subset \mathbb{Q}^n$ be an integral reflexive polytope. Write 
\[
P = P_1 \oplus \dots \oplus P_r
\]
where each $P_i$ is a lattice polytope containing the origin and $\oplus$ denotes the Minkowski sum. Assume no $P_i$ is equal to the origin. Then, we have $r\leq 2n$. Furthermore, the equality holds if and only if $P$ is an $n$-dimensional crosspolytope, i.e., the dual of an $n$-dimensional cube.
\end{proposition}

The $n$-dimensional cube is precisely the moment polytope of $(\mathbb{P}^1)^n$. 
Thus, Proposition~\ref{prop:cube} is exactly Theorem~\ref{introthm:charct-p1n} restricted to the setting of Fano toric varieties.
Thus, Theorem~\ref{mukai type gklt} also has a consequence in the realm of reflexive polytopes.
Indeed, we show the following corollary:

\begin{corollary}\label{introcor:poly}
Let $P\subset \mathbb{Q}^n$ be an integral reflexive polytope with $f$ facets. Write 
\[
P= \lambda_1 P_1 \oplus \dots \oplus \lambda_r P_r
\]
where each $P_i$ is a lattice polytope containing the origin, $\lambda_i \geq 0$, and $\oplus$ denotes the Minkowski sum. Assume that no $P_i$ is equal to the origin. Then, we have that $\sum_{i=1}^r \lambda_i \leq f$. Furthermore, the equality holds if and only if $P$ is a product of simplices.
\end{corollary}

Finally, we prove two more statements that are related to Mukai's Conjecture \ref{mukai conjecture} and to characterizations of products of projective spaces.

Firstly, we consider smooth Fano varieties with many Mori fiber contractions.
More precisely, we consider those smooth Fano varieties whose nef cone coincides with the pseudo-effective cone.
In particular, we show the following theorem, which answers positively a question by J. Starr~\cite{sta}:

\begin{theorem}[{cf.~Theorem~\ref{simpliciality}}]\label{simpliciality-intro}Let $X$ be a smooth Fano variety such that every elementary contraction is of  fiber type. Then the cone $\mathrm{Nef}(X)$ is simplicial.
\end{theorem}

We also prove Conjecture \ref{mukai conjecture} for a smooth Fano variety such that every elementary contraction is of fiber type (Theorem~\ref{many MFS}). 
Finally, we give a new proof of the fact that smooth images of products of projective spaces are again products of projective spaces. 

\begin{theorem}[{cf.~Theorem~\ref{thm:smth-imgs}}]
Let $X$ be a smooth projective variety 
which is the image of a product of projective spaces. Then, $X$ itself is a product of projective spaces.
\end{theorem}

The previous statement can also be deduced from the work~\cite{DHP08} of Demailly, Hwang, and Peternell, characterizing complex manifolds that are the image of a complex torus via surjective morphisms.

\medskip

\subsubsection*{Structure of the paper}
In Section~\ref{preparation}, we collect some basic notions and notations that will be used throughout this article.
In Section~\ref{fujita type:section}, we prove the numerical Fujita-type theorems and a characterization of projective spaces of Kobayashi--Ochiai type.
In Section~\ref{induction:proof}, we prove Conjecture \ref{Mukai type conjecture} by induction on the dimension.
In Section~\ref{sec:total-index-dP}, we compute the total indices of smooth del Pezzo surfaces.
In Section~\ref{sec:smooth-Fano-3folds}, we give a classification of smooth Fano threefolds with nef complexity at most one. 
In Section \ref{many MFS:section}, we study smooth Fano varieties with many Mori fiber space structures.
In Section~\ref{sec:smooth-images}, we study smooth images of products of projective spaces.

\section*{Acknowledgements}
This paper is an outcome of the workshop ``Higher dimensional log Calabi--Yau varieties'' held in 2024 at the American Institute of Mathematics. 
The authors are indebted to the the institute for its support and hospitality.

The authors thank Professors Cinzia Casagrande, Kento Fujita, and Osamu Fujino for valuable comments and discussions related to the results of the paper. In particular, Fujino suggested a simpler proof of Theorem 
\ref{numerical Fujita type}.
\medskip

SF was partially supported by the ERC starting grant $\#$804334 and by Duke University.
YG was partially supported by grants JSPS KAKENHI $\#$16H02141, 17H02831, 18H01108, 19KK0345, 20H0011.
JM was partially supported by NSF research grant DMS-2443425.
RS was partially supported by 
“Programma per giovani ricercatori Rita Levi Montalcini” of MUR and by PSR 2022 – Linea 4 of the University of Milan. 
He is a member of the GNSAGA group of INDAM. CW was partially supported by starting grant $\#$53331005425 of Tsinghua University.
KW was partially supported by JSPS KAKENHI $\#$21K03170, 25K06940, and part of this work was carried out during his stay at Universit\'{e} C\^{o}te d’Azur. 

\section{Preliminaries}
\label{preparation}

In this section, we recall some notation and basic notions that will be used throughout the article.

\subsubsection*{Big and pseudo-effective Weil divisors}
We recall the definition of bigness and pseudo-effectiveness for Weil and $\mathbb R$-Weil divisors.

\begin{definition}
\label{def:pseudoeffective.Weil}
Let $X$ be a normal projective variety and let 
$D$
(resp. $D'$)
be a 
Weil 
$\mathbb Q$-divisor
(resp. $\mathbb R$-divisor).
\begin{enumerate}
    \item 
We say that 
$D$ 
is big if there exist an ample 
$\mathbb Q$-Cartier 
$\mathbb Q$-divisor 
$A$ 
and an effective 
$\mathbb Q$-divisor 
$E$ 
such that 
$D \sim_\mathbb{Q} A+E$.
\item 
We say that 
$D$ 
is pseudo-effective if for every ample 
$\mathbb Q$-Cartier 
$\mathbb Q$-divisor 
$H$, 
$D+H$ 
is big.
\item 
We say that 
$D'$ 
is big if there exist an ample 
$\mathbb R$-Cartier 
$\mathbb R$-divisor 
$A'$ 
and an effective 
$\mathbb R$-divisor 
$E'$ 
such that 
$D' \sim_\mathbb{R} A'+E'$.
\item 
We say that 
$D'$ 
is pseudo-effective if for every ample 
$\mathbb R$-Cartier 
$\mathbb R$-divisor 
$H'$, 
$D'+H'$ 
is big.
\end{enumerate}
\end{definition}

\begin{remark}
\label{rem:def.pseudoeff}
    We adopt the same notations and assumptions as in the above definition, and we recall some of its simplest consequences.
\begin{enumerate}
    \item
    \label{cond:rem.pseff.eff.pseff}
    If a positive integral multiple of 
    $D$ 
    is effective, then
    $D$ 
    is pseudo-effective.
    \item 
    \label{cond:rem.pseff.nef.big}
    If 
    $D$ 
    is moreover a nef 
    $\mathbb Q$-Cartier 
    $\mathbb Q$-divisor, 
    then $D$ is big if and only if 
    $D^{\dim X} >0$, 
    see 
    \cite[Theorem 2.2.15]{laz-pos-1}.
    Moreover the positivity of 
    $D^{\dim X}$ 
    (or lack thereof) can be computed for 
    $\pi^\ast D$ 
    via any generically finite map 
    $\pi \colon Y \to X$.
     \item 
    \label{cond:rem.pseff.equiv.def.big}
    The $\mathbb R$-divisor $D'$ is big if and only if 
    $D' \sim_{\mathbb R} \sum_{i=1}^n a_i D'_i$, 
    where
    $a_i > 0$ 
    and 
    $D'_i$
    is a big 
    $\mathbb Q$-divisor
    for all 
    $1\leq i\leq n$.
    \\
    Let us prove the claimed equivalence. 
    If 
    $D' 
    \sim_{\mathbb R} 
    \sum_{i=1}^n a_i D'_i$, 
    then 
    \begin{align*}
    D' 
    \sim_{\mathbb R} 
    \sum_{i=1}^n a_i (A'_i + E'_i) 
    =
    \tilde A' + \tilde E', 
    \quad 
    \tilde A' 
    \coloneqq
     \sum_{i=1}^n a_i A'_i, 
     \;
     \tilde E'\coloneqq \sum_{i=1}^n a_i E'_i.
    \end{align*}
    Vice versa, if 
    $D'$ 
    is big, then 
    $D' \sim_{\mathbb R} A'+E'$
    and 
    $A' = \sum_{j=1}^s b_j A'_j$, 
    $E' = \sum_{k=1}^t c_k E'_k$,
    where 
    $b_j, c_k >0$,
    $A'_j$ 
    (resp. $C'_k$)
    ample 
    (resp. effective)
    $\mathbb{Q}$-divisors.
    If 
    $t \leq s$, 
    then it suffices to choose positive rational numbers
    $\tilde b_k, \tilde c_k$
    for $1 \leq k \leq t$
    such that 
    $0 < c_k -\tilde c_k < b_k -\tilde b_k$
    and write
    \begin{align*}
    A'+E'
    &
    \sim_\mathbb{R}
    \sum_{k=1}^s 
    \underbrace{
    \left (
    \tilde b_k A'_k + \tilde c_k E'_k
    \right )}_{\text{big $\mathbb Q$-divisor}}
    +
    \sum_{k=1}^s 
    \overbrace{
    \left(
    c_k - \tilde c_k
    \right)}^{>0}
    \underbrace{
    \left (
    A'_k + E'_k
    \right )}_{\text{big $\mathbb Q$-divisor}}
    \\
    & +
    \sum_{k=1}^s 
    \overbrace{
    \left(
    b_k - \tilde b_k
    - c_k + \tilde c_k
    \right)}^{>0}
    \underbrace{A'_k}_{\text{ample $\mathbb Q$-divisor}}
    +
    \sum_{j > s}
    b_j A'_j.
    \end{align*}
    When $t < s$, it suffices to replace each 
    $A'_j$ 
    with
    \[
    A'_j = \underbrace{
    \frac 1n A'_j + 
    \frac 1n A'_j + 
    \dots +
    \frac 1n A'_j
    }_{\text{$n$ times}}
    \]
    and run the previous part of the argument.
\end{enumerate}
\end{remark}

\subsubsection*{Generalized pairs}
We recall now the basic notions about generalized pairs. 
For more details, we refer the reader to \cite{Bir20}.
First, we recall the definition of generalized log canonical and Kawamata log terminal.

\begin{definition}\label{defi:gp}
A \textit{generalized pair} $(X,B+M)$ is the datum of
\begin{itemize}
\item a normal projective variety $X$, 
\item an effective $\mathbb{R}$-divisor $B$ on $X$, and 
\item a b-$\mathbb{R}$-Cartier b-divisor $M$ over $X$ represented by some projective birational morphism $\varphi \colon  X' \to X$ and 
a nef $\mathbb{R}$-Cartier divisor $M'$ on $X'$
\end{itemize}
such that $M = \varphi_* M'$ and $K_{X}+B+M$ is $\mathbb{R}$-Cartier.
\end{definition}

\begin{definition}
\label{def:gen-lc}
Let $(X, B+M)$ be a generalized pair.
Let $\pi \colon  Y\rightarrow X$ be a projective birational morphism.
We can write
\[
K_Y+B_Y+M_Y=\pi^*(K_X+B+M),
\]
for some divisor $B_Y$ on $Y$.
The {\em generalized log discrepancy}
of $(X,B+M)$ at a prime divisor $E\subset Y$, denoted by $a_E(X,B+M)$,
is defined to be $1-{\rm coeff}_E(B_Y)$.

A generalized pair $(X, B+M)$ is said to be {\em generalized log canonical} (or glc for short) if all of its generalized log discrepancies are non-negative.
A generalized pair $(X, B+M)$ is said to be {\em generalized Kawamata log terminal} (or gklt for short) if all of its generalized log discrepancies are positive.
In the previous definitions, if the analogous statement holds for $M=0$, then we drop the word ``generalized''.
\end{definition}

\subsubsection*{Log Calabi--Yau pairs}
In this section we introduce the definition of the log Calabi--Yau property for generalized pairs and we state several technical lemmata that shall be used in this article.

\begin{definition}
A generalized pair $(X,B+M)$ is said to be {\em generalized log Calabi--Yau} if
$K_X+B+M\equiv 0$ and $(X,B+M)$ has glc singularities.
We say that $(X,B+M)$ is a {\em generalized klt log Calabi--Yau pair} if 
$K_X+B+M\equiv 0$ and $(X,B+M)$ has generalized klt singularities.

When $M=0$, then we simply talk about a log Calabi--Yau pair 
(resp. klt log Calabi--Yau pair).
\end{definition}

Recall that a normal projective variety $X$ is said to be of Fano type if it admits an effective 
$\mathbb Q$-divisor
$\Theta$
such that 
$(X, \Theta)$
is klt and 
$-(K_X+\Theta)$
is ample.

The following is a simple characterization of the Fano type property using klt log Calabi--Yau pairs.
\begin{lemma}\cite[\S 2.10]{Bir19}
    A normal projective variety $X$ is Fano type if and only if there exists an effective $\mathbb R$-divisor $\Gamma$ such that $(X,\Gamma)$ is klt log Calabi--Yau and $\Gamma$ is big.
\end{lemma}

We can straightforwardly generalize the above criterion to the case of generalized log pairs.

\begin{lemma}
\label{lemma:fano_type}
    Let 
    $(X, B+M)$
    be a projective 
    generalized klt log Calabi--Yau pair.
    Then $X$ 
    is of Fano type if and only if 
    $-K_X$
    is big.
\end{lemma}

\begin{proof}
    If $X$ is of Fano type it follows immediately from the definition that $-K_X$ is big.
    We now show that the opposite implication also holds.
    
    The claim is unaffected by taking a small $\mathbb Q$-factorialization.

    By assumption, $B+M$ is big, so we may write $B+M \sim_{\mathbb R} A +E$, where $A$ is ample and $E \geq 0$.
    Then, for $0 < \epsilon \ll 1$, $(X,(1-\epsilon)(B + M)+\epsilon E)$ is gklt.
    Then, since $A$ is ample and $M$ is the pushforward of a nef divisor, $\epsilon A + (1-\epsilon)M$ is the pushforward of a nef and big divisor.
    Thus, we may find $0 \leq D \sim_{\mathbb R} \epsilon A+(1-\epsilon)M$ such that $(X,(1-\epsilon)B+\epsilon E + D)$ is klt.
\end{proof}

We also show that the condition of being of Fano type descends along any contraction morphism for projective generalized klt log Calabi--Yau pairs.

\begin{lemma}
\label{big.anticanonical.lemma}
Let $(X, D)$ be a projective klt log Calabi--Yau pair.
Let $f \colon X \to Y$ be a contraction.
If $X$ is Fano type, then also $Y$ is Fano type.
\end{lemma}

\begin{proof}
As we are assuming that $X$ is Fano type, we can and will assume that $D$, equivalently, $-K_X$ is big, cf. Lemma~\ref{lemma:fano_type}.

The canonical bundle formula \cite{Amb04, Amb05} implies that there exists an effective divisor $E$ on $Y$ such that $(Y, E)$ is klt and 
$0 \equiv K_X+D \sim_{\mathbb R} f^\ast(K_Y+E)$.
Hence, by Lemma~\ref{lemma:fano_type}, it suffices to show that $-K_Y$ is big.

If $f$ is birational, then $-K_Y$ is big since pushing forward preserves bigness for divisors.

Let us assume that $f$ has strictly positive relative dimension.
By \cite[Corollary 1.4.3]{bchm}
there exists a $\mathbb Q$-factorialization $\psi \colon Y' \to Y$;
we define $E'  \coloneqq  \psi^{-1}_\ast E$.
Then \cite[Proposition 3.6]{BDCS20} implies that there exists a $\mathbb Q$-factorial klt pair
$(X', D')$ isomorphic in codimension one to $(X, D)$ and a commutative diagram
\[
\xymatrix{
X' \ar[r] \ar[d] & X \ar[d] 
\\
Y' \ar[r] & Y.}
\]
As $D'$ is the strict transform of $D$ on $X$, $D'$ is big.
Hence, we can conclude from \cite[Theorem 3.1]{fujgo} that $-K_{Y'}$ is big which in turn implies that $-K_{Y}$ is also big.
\end{proof}

Finally, we prove a lemma showing that the Cartier condition descends for divisors that are numerically trivial along any contraction of a projective generalized klt log Calabi--Yau pair.

\begin{lemma}
\label{lemma.pullback}
    Let 
    $(X, \Delta+M)$
    be a projective 
    generalized klt log Calabi--Yau pair.
    Assume that 
    $-K_X$
    is big.
    Let 
    $f \colon X \to T$
    be a contraction and let
    $D$ 
    be a Cartier divisor.
    If 
    $D \equiv_f 0$, 
    then there exists a Cartier divisor 
    $D'$
    on 
    $T$
    such that 
    $D \sim f^\ast D'$.
\end{lemma}

\begin{proof}
   By Lemma~\ref{lemma:fano_type}, $X$ is of Fano type.
   Fix a boundary $B$ such that $(X,B)$ is klt and $-(K_X+B)$ is nef and big.
   Then, the claim follows by applying the relative basepoint-free theorem \cite[Theorem 3.24]{komo} to $T$, $(X,B)$, and the morphism $f$.
\end{proof}

\section{Numerical Fujita-type theorems}\label{fujita type:section}

In this section, we discuss the numerical Fujita-type
theorems and a characterization of projective spaces of Kobayashi--Ochiai type.
First, we prove the following theorem:

\begin{theorem}
\label{numerical Fujita type}
Let $X$ be an $n$-dimensional normal projective variety. 
Let 
$L_{1}, \dots, L_{k}$ 
be nef and big Cartier divisors on $X$, and 
$a_1, \dots, a_k$
positive real numbers.
If 
$\sum_{i=1}^{k}a_i \geq n+1$, 
then 
$K_X+a_1L_1+\dots + a_{k}L_{k}$ 
is pseudo-effective.  
\end{theorem}

The following is a weak version of the above theorem which will be used in its proof.

\begin{proposition}\label{weak numerical Fujita type}
Let 
$X$ 
be an 
$n$-dimensional 
normal projective variety and 
$L$ 
a nef and big Cartier divisor on 
$X$.
Then, 
$K_X+(n+1)L$ 
is pseudo-effective. 
\end{proposition}

\begin{proof}
Let us consider a resolution 
$\pi \colon X' \to X$.
It suffices to show that there exists a positive integer 
$0 < m \leq n+1$
such that 
$K_{X'}+m\pi^\ast L$ 
is effective:
indeed, 
\begin{align*}
\pi_\ast (K_{X'}+m\pi^\ast L)
=K_X+mL,  
\end{align*}
and the pushforward of an effective Weil divisor is still effective.
Thus, the conclusion will follow from Remark
\ref{rem:def.pseudoeff}.\ref{cond:rem.pseff.eff.pseff} 
which implies that
\begin{align*}
    K_X+(n+1)L =
    \underbrace{(K_X+mL)}_{\text{effective}}+
    \overbrace{(n+1-m)}^{\geq 0}
    \underbrace{L}_{\text{big and nef}}
\end{align*}
is pseudo-effective. 
Hence, substituting $(X, L)$ with $(X', \pi^\ast L)$, we can assume that $X$ is smooth and proceed to prove that an integer $m$ satisfying the above properties exists.

Let us consider the numerical polynomial 
$P(t) \coloneqq \chi ( K_X+tL)$.
By the Hirzebruch--Riemann--Roch Theorem, 
\begin{align*}
P(t) = \dfrac{L^n}{n!}t^n+O(t^{n-1}).
\end{align*}
Since 
$L^n>0$, 
$P(\tilde t) >0$ 
for all 
$\tilde t \gg 0$.
By the Kawamata--Viehweg vanishing theorem,
for all positive integers
$t$, 
$P(t)= 
\dim H^0(X, K_X+tL)$.
Since 
$P(t)$ 
is a polynomial in
$t$
of degree 
$n$,
there exists a positive integer 
$0 < m \leq n+1$ 
such that  
$P(l) \not =0$.
\end{proof}

We now prove Theorem \ref{numerical Fujita type}.

\begin{proof}
Up to substituting each $a_i$ with $a_i':=(n+1)\frac{a_i}{\sum_{i=1}^k a_i} \leq a_i$, we may assume that $\sum_{i=1}^ka_i=n+1 $.
Since 
\begin{align*}
\sum_{i=1}^k a_i(K_X+(n+1)L_i)= (n+1)(K_X+\sum_{i=1}^ka_i L_i),
\end{align*}
then Proposition \ref{weak numerical Fujita type} and Remark \ref{rem:def.pseudoeff}.\ref{cond:rem.pseff.eff.pseff} 
imply that 
$K_X+\sum_{i=1}^ka_i L_i$ 
is pseudo-effective.
\end{proof}

We readily obtain the following corollaries:

\begin{corollary}\label{numerical Fujita type-cor}
Let $X$ be an $n$-dimensional normal projective variety. 
Let 
$L_{1}, \dots, L_{k}$ 
be nef and big Cartier divisors on $X$, and 
$a_1, \dots, a_k$
positive real numbers.
If 
$\sum_{i=1}^{k}a_i > n+1$, 
then 
$K_X+a_1L_1+\dots + a_{k}L_{k}$ 
is big.   
\end{corollary}

\begin{proof}
By Theorem \ref{numerical Fujita type}, $K_X+ a_1L_1+\dots + (a_k -\epsilon) L_{k}$ is pseudo-effective for all  $0<\epsilon \ll a_k$.
Since $L_{k}$ is big, the conclusion now follows from Remark \ref{rem:def.pseudoeff}.
\end{proof}

Moreover, we prove the following theorem which is a
characterization of projective spaces of Kobayashi--Ochiai type:

\begin{theorem}
\label{Fano-case}
Let 
$X$ 
be a normal projective variety of dimension 
$n$. 
Let 
$L_1, \dots, L_k$ 
be nef and big Cartier divisors on 
$X$, 
and 
$a_1, \dots, a_k$
positive real numbers.

If there exists a decomposition 
$-K_X \equiv \sum_{i=1}^k a_i L_i$ 
such that 
$\sum_{i=1}^k a_i\geq n+1$, then equality holds, i.e., $\sum_{i=1}^k a_i = n+1$. 
Moreover, $X$ is isomorphic to the projective space $\mathbb P ^n$ and 
for all $1\leq i\leq k$, $\mathcal O_X(L_i)$ is isomorphic to $\mathcal O_{\mathbb P^n}(1)$.
\end{theorem}

\begin{proof}
By Theorem~\ref{numerical Fujita type}, it follows that 
$\sum_{i=1}^k a_i=n+1$ 
and that
\begin{align*}
0\equiv (n+1)(K_X+\sum_{i=1}^ka_iL_i)= \sum_{i=1}^k a_i(K_X+(n+1)L_i).    
\end{align*}
By Proposition \ref{weak numerical Fujita type},  $K_X+(n+1)L_i$ is pseudo-effective. Thus, we conclude that  $K_X+(n+1)L_i \equiv 0$ for all $i$. 

We now claim that 
$(X, 0)$
is canonical.
First, the existence of the decomposition 
$-K_X\equiv \sum_{i=1}^ka_iL_i$
implies that 
$K_X$
is 
$\mathbb Q$-Cartier.
Moreover, 
if $(X, 0)$ had worse than canonical singularities, then taking first a dlt modification 
$\psi_1 \colon X' \to X$
of 
$(X, 0)$, 
followed up by a terminalization 
$\psi_2 \colon X'' \to X'$
of 
$(X', (1-\epsilon)\mathrm{Exc}(\psi_1))$,
for some $0< \epsilon \ll 1$, 
we would obtain a partial resolution 
$\psi \coloneqq \psi_2\circ \psi_1 \colon X'' \to X$
together with an effective 
$\mathbb R$-divisor 
$G$
on 
$X''$
such that
\[
K_{X''}+ \psi^\ast \sum_{i=1}^ka_iL_i +G
= \psi^\ast(K_X+\sum_{i=1}^ka_iL_i) \equiv 0.
\]
But then $G=0$, which implies that $(X, 0)$ has canonical singularities, otherwise that would contradict Proposition~\ref{weak numerical Fujita type} applied to 
$K_{X''}+  \sum_{i=1}^ka_i\psi^\ast L_i$.

Hence $(X, 0)$ is canonical and 
$K_X+(n+1)L_i \equiv 0$ for all $i$.
In particular, that implies that all $L_i$ are semiample and $X$ is weak log Fano;
thus, $\chi(\mathcal O_X)=1$.
We can then consider the numerical polynomials
\[
P_i(t) \coloneqq \chi(X, -tL_i),
\]
see \cite[Corollary 3]{fulton.hrr}.
The polynomials
$P_i(t)$ 
have degree 
$n$;
moreover, 
for any integer $l$, 
$K_X+lL_i \equiv (l-n-1)L_i$.
Thus, the Kawamata--Viehweg vanishing theorem implies that 
\begin{align*}
    \chi(X, -lL_i)=
    (-1)^n\chi(X, K_X+lL_i)=
    \begin{cases}
        1 
        &
        \text{for }
        l=0, 
        \\
        0
        &
        \text{for }
        1\leq l\leq n.
    \end{cases}
\end{align*}

Hence, as
$P_i(l)=\chi(\mathbb P^n, \mathcal O_{\mathbb P^n}(-l))$, 
for all $0 \leq l \leq n$, 
then 
$P_i(t)=\chi(\mathbb P^n, \mathcal O_{\mathbb P^n}(-t))$
as numerical polynomials in the indeterminate $t$.
Finally, 
\cite[Theorem 1.1]{kooc}
implies the sought conclusions.
\end{proof}

Finally, we prove also a birational version of the previous result which holds generalized pairs.

\begin{corollary}\label{generalized complexity-cor}
Let $(X,B+M)$ be an $n$-dimensional generalized pair such that $K_X+B+M\equiv 0$.
Assume that the b-divisor $M$ descends to a $M_{X'}$ on a higher birational model $\varphi \colon  X' \to X$.

If there exists a decomposition $M'= \sum_{i=1}^{k} a_i M'_i+N'$ such that for all $1\leq i \leq k$, $a_i \in \mathbb{R}_{>0}$ and $M'_i$ is a nef and big Cartier divisor, and $N'$ is a nef $\mathbb{Q}$-Cartier divisor, then $\sum_{i=1}^{k}a_{i} \leq n+1$.
Moreover, if the equality holds, then  $B=0$ and $N'\equiv 0$. 
\end{corollary}

\begin{proof}
We apply Corollary \ref{numerical Fujita type-cor} to $X'$. Assume that $\sum_{i=1}^{k}a_{i} >n+1$.
We see that $\varphi^*(K_X+B+M)=K_{X'}+B'+M'\equiv 0$.
Take effective divisors $B'^+, B'^-$ without common components such that $B'=B'^+-B'^-$.
Then, we see that $B'^-$ is $\varphi$-exceptional.

By Corollary \ref{numerical Fujita type-cor}, $K_{X'}+B'^+ +M'$ is big.
On the other hand, by construction, we have $K_{X'}+B'^+ +M' \equiv B'^-$, which is $\varphi$-exceptional.
This is a contradiction.

Now assume that $\sum_{i=1}^{k}a_{i} = n+1$.  
Since $K_{X'}+ \sum_{i=1}^{k} a_i M'_i$ is pseudo-effective by Theorem \ref{numerical Fujita type} and $K_X + B + M \equiv 0$ by definition, we have  
\[
\varphi_* \left(K_{X'} + \sum_{i=1}^{k} a_i M'_i\right) \equiv 0, \quad B = 0, \quad \text{and} \quad \varphi_*N' \equiv 0.
\]  
In particular, it follows from the negativity lemma that $N' \equiv 0$.
\end{proof}

\section{Proof of  Conjecture \ref{Mukai type conjecture}}\label{induction:proof} 

Before discussing a proof of Conjecture 
\ref{Mukai type conjecture}, 
we first introduce some notation.

Let us recall that for a normal projective variety 
$X$
we denote
$\rho(X) \coloneqq \dim_{\mathbb R} \NS(X)$, 
where 
$\NS(X)$ 
denotes the 
$\mathbb R$-vector 
space of 
$\mathbb R$-Cartier 
divisors on 
$X$ 
modulo numerical equivalence.

\begin{definition}
Let $X$ be a normal variety.
Let $D_1, \dots, D_s$ be $\mathbb R$-Cartier
divisors
on 
$X$.
We define
\begin{align*}
&\Sigma(D_1, \dots, D_s)
 \coloneqq 
{\rm span}([D_1], \dots, [D_s])
\subset \NS(X),
\quad
\text{and}
\\
&\sigma(D_1, \dots, D_s)  \coloneqq  \dim_\mathbb{R} \Sigma(D_1, \dots, D_s).
\end{align*}
\end{definition}

We are now ready to prove our main result of this section, which is a refined version of 
Conjecture
\ref{Mukai type conjecture}.

\begin{theorem}
\label{gen.mukai.thm}
Let $X$ be a 
normal projective variety.
Assume that there exists a generalized klt log Calabi--Yau pair
$(X, \Delta+M)$
satisfying the following properties:
\begin{enumerate}
\item 
\label{condition:nef_cartier}
$M=\sum_{i=1}^k m_iM_i$, 
and for all 
$i=1, \dots, k$, 
$M_i$ 
is nef and Cartier and $M_i \not\equiv 0$;
and
\item 
\label{condition:big}
$\Delta+\sum_{i=1}^k m_iM_i$
is big.

\end{enumerate}
Then,
\begin{align}
\label{gen.mukai.thm.eqn}
\sum_{i} m_i \leq 
\min (\dim X+\sigma(M_1 \dots, M_k), 2\dim X).
\end{align}
Moreover, if equality holds in \eqref{gen.mukai.thm.eqn},
then the following properties hold:
\begin{enumerate}[label=(\roman*)]
\item 
\label{gen.mukai.thm.2nd.ineq}
$\dim X + \sigma(M_1, \dots, M_k) \leq 2\dim X$;

\item 
\label{gen.mukai.thm.cond.=}
$X$  
is a product of projective spaces, 
$X=\prod_{j=1}^l \mathbb P^{t_j}$.
We denote by $pr_{j}$
the projection onto the 
$j$-th 
factor of 
$X$;
\item 
\label{gen.mukai.thm.cond.delta.null}
$\Delta=0$ 
and 
$K_X+
\sum_{i=1}^k m_iM_i \equiv  0$;
\item 
\label{gen.mukai.thm.cond.boundary}
for all 
$i=1, \dots, k$,
there exists 
$\bar{j}_i \in \{1, \dots, l\}$
such that 
$M_i = {\rm pr_{\bar{j}_i}}^\ast H_{\bar{j}_i}$, 
where 
$H_{\bar{j}_i}$
denotes the hyperplane section of the 
$\bar{j}_i$-th 
factor of 
$X$.
\end{enumerate}
\end{theorem}

\begin{remark}
\label{rem.eq.gen.mukai.thm}
It is an immediate consequence of the above theorem that if equality holds in 
\eqref{gen.mukai.thm.eqn}, 
properties 
\ref{gen.mukai.thm.cond.=}
and 
\ref{gen.mukai.thm.cond.boundary}
in the above statement together imply that 
$M_1 \dots, M_k$ 
generate 
$\NS(X)$.

Moreover, if equality holds also in 
\ref{gen.mukai.thm.2nd.ineq}, 
then 
\ref{gen.mukai.thm.cond.=} readily implies that
$X \simeq (\mathbb P^1)^{\dim X}$.
\end{remark}

\begin{proof}
The theorem holds for $\dim X=1$: indeed, in that case $X \simeq \mathbb P^1$ and
\[
\sum_i m_i \leq 2=\dim X + \rho(X)=2\dim X.\]
Hence, we can and shall assume that $\dim X \geq 2$.

For the reader's convenience, we divide the proof into steps.

\medskip

{\bf Step 0.}
{\it In this step, we show that any nef $\mathbb Q$-Cartier divisor on $X$ is semiample.
We also show that such result holds also in the relative case.}

Lemma~\ref{lemma:fano_type} implies that 
$X$ 
is of Fano type; thus, any 
small $\mathbb Q$-factorialization 
$\pi \colon \tilde X \to X$
of 
$X$
is a Mori dream space.
In particular, for any nef $\mathbb Q$-Cartier divisor $N$ on $X$, $\pi^\ast N$ is semiample, which in turn implies that $N$ itself is semiample on $X$, cf. \cite[Lemma 2.1.13]{laz-pos-1}.

The same argument also works in the relative setting:
given a morphism $X \to Z$ and a relatively nef $\mathbb Q$-Cartier divisor $N'$ over $Z$, then $N'$ is relatively semiample over $Z$.

\medskip 

{\bf Step 1}.
{\it 
In this step, we show that 
the theorem holds when 
$\sigma(M_1, \dots, M_k)=1$, independently of the dimension of $X$.
}

Since 
$\sigma(M_1, \dots, M_k)=1$,
then 
$\dim X + \sigma(M_1, \dots, M_k) < 2\dim X$, 
as we have assumed that 
$\dim X \geq 2$.
Moreover, we are free to assume that the divisors 
$M_i$
are all proportional to each other in $\NS(X)$.
Hence, up to reordering the indices, we may assume that 
$K_X+\Delta+sM_1\equiv 0$, 
for some 
$s \geq \sum_{i=1}^k m_i$
and equality holds if and only if 
for all 
$1 \leq i,j \leq k$,
$M_i \equiv M_j$.
On the other hand, since $X$ 
is of Fano type, 
then 
$X$
is rationally connected;
thus
$h^1(X,\mathcal O_X)=0$, 
and
$\chi(X, \mathcal O_X)=1$, so that all
$M_i$
are linearly equivalent.

We now assume that 
$\sum_{i=1}^k m_i \geq  \dim X+1$ and
proceed to show that equality holds in the previous two inequalities and that conditions
\ref{gen.mukai.thm.cond.=}-\ref{gen.mukai.thm.cond.boundary}
in the theorem are satisfied.
It follows that
$s \geq \dim X+1$, 
and that
for $0 \leq l \leq \dim X$,	
\begin{align*}
-lM_1 \equiv K_X+\Delta+(s-l)M_1
\quad
\text{and}
\quad 
s-l >0.
\end{align*}

If 
$M_1$
is big,
we can proceed exactly as in the final part of the proof of Theorem~\ref{Fano-case}:
the Kawamata--Viehweg vanishing theorem implies that 
\begin{align*}
    \chi(X, \mathcal O_X(-lM_1))=
    \begin{cases}
        1 
        &
        \text{for }
        l=0, 
        \\
        0
        &
        \text{for }
        1\leq l\leq \dim X.
    \end{cases}
\end{align*}
It then follows that
$P(l) \coloneqq \chi(X, \mathcal O_X(-lM_1))$
is a numerical polynomial of degree $\dim X$,
and 
$P(l) \neq 0$
for 
$l\gg 1$, 
as
$M_1$
is big and nef.
Then 
$P(l)=\chi(\mathbb P^{\dim X}, \mathcal O(-l))$,
as these two polynomials attain the same values for all $\dim X+1$ integral values of 
$l \in [0, \dim X]$.
As
$M_1$
is big and semiample, see Step 0, there exists a birational contraction 
\[\gamma \colon X \to X^\circ \coloneqq 
{\rm Proj}
\left(
\sum_{l\geq 0}
H^0(X,\mathcal O_X(lM_1))
\right).
\]
As
$P(l)=\chi(\mathbb P^{\dim X}, \mathcal O(-l))$
and 
$M_1$ is nef and big, 
then
$M_1^{\dim X}=1$
and
$P(-1)=h^0(X,\mathcal O_X(M_1))=
h^0(\mathbb P ^{\dim X},\mathcal O(1))=\dim X+1$.
Furthermore, Lemma~\ref{lemma.pullback} implies that there exists an ample Cartier divisor
$L$
on 
$X^\circ$
such that 
$M_1 \sim \gamma^\ast L$.
Hence, 
$L^{\dim X}=M_1^{\dim X}=1$
and 
$h^0(X^\circ, \mathcal O_{X^\circ}(L))=
h^0(X, \mathcal O_X(M_1))
=\dim X+1$.
By
\cite[Theorem~1.1]{kooc}, 
$X^\circ \simeq \mathbb P ^{\dim X}$, 
$L$
is a hyperplane section.
Hence, 
$s =\dim X+1 \geq \sum_{i=1}^k m_i$,
since, as $\gamma$ is a crepant birational map
for the generalized pair $(X,\Delta+(\dim X+1)M_1)$, we have 
\begin{align}
\label{equiv:crepant.gamma}
   0\equiv 
   K_X+\Delta+(\dim X+1)M_1
   \equiv 
   \gamma^\ast 
    (K_{\mathbb P^{\dim X}}+(\dim X+1)L).
\end{align}
We consider
$(\mathbb P^{\dim X}, 0 + (\dim X+1)L)$
as a generalized pair with 
boundary
$0$
and moduli b-divisor given by the Cartier closure of 
$(\dim X+1)L$.
This pair has generalized terminal singularities and 
\eqref{equiv:crepant.gamma} 
implies that 
$\gamma$ 
is crepant for the generalized pairs
\[
(X, \Delta + (n+1)M_1), 
\quad
(\mathbb P^{\dim X}, 0 + (\dim X+1)L).
\]
Hence, 
$\Delta=0$, 
and 
$\gamma$ 
is an isomorphism
since 
$(\mathbb P^{\dim X}, 0 + (\dim X+1)L)$ 
is $\mathbb{Q}$-factorial terminal.

If 
$M_1$
is not big, 
we can consider 
the Iitaka fibration 
\[
f \colon X \to T, \quad T \coloneqq \mathrm{Proj}(\oplus_{l \geq 0} H^0(lM_1))
\]
of 
$M_1$; 
in particular, $\dim T < \dim X$.
Thus
$M_1\sim f^\ast L_1$, 
for a suitable ample Cartier divisor 
$L_1$
on 
$T$, by
Lemma~\ref{lemma.pullback}.
Then, by the canonical bundle formula, on $T$ there exists 
$\Delta_T$ 
such that 
$(T, \Delta_T)$
is klt and 
$K_T+\Delta_T+sL_1 \sim_{\mathbb R} 0$.
As 
$s\geq n+1 > \dim T+1$, 
this contradicts 
Theorem~\ref{Fano-case}; 
thus, 
$M_1$ 
can only be big and we are done by the previous paragraph.

\medskip

From here onwards, we shall work by induction on 
$(\dim X, \rho(X))$, 
using the lexicographic order.

Step $1$ implies that the conclusion of the theorem holds when 
$\rho(X)=1$
completing the proof of the base of the induction -- the case 
$\dim X=1$ 
was dealt with at the start of the proof.

\medskip

The remainder of the proof will focus on proving the inductive step.

\medskip

{\bf Step 2}.
{\it
In this step, we show that if
all 
$M_i$ 
are nef and big, then the theorem holds}.

This is an immediate consequence of Corollary
\ref{generalized complexity-cor}
and Theorem
\ref{Fano-case}:
if 
\begin{align}
\label{gen.mukai.thm.ineq.allbig}
\sum m_i \geq \minimo \geq \dim X+1,    
\end{align}
then the two results just cited imply that equality must hold everywhere in 
\eqref{gen.mukai.thm.ineq.allbig}
and, moreover, 
$X\simeq \mathbb P^{\dim X}$, and $\Delta =0$.

\medskip

By Step 2, we can then assume that at least one of the 
$M_i$, 
say 
$M_1$, 
is nef and not big.
In view of this, we choose
$G_1$ 
to be a face of 
${\rm Nef}(X)$ 
of maximal dimension among those that contain 
$M_1$
and such that the contraction 
${\rm cont}_{G_1} \colon X \to Y_{G_1}$
associated to 
$G_1$
is of fiber type, i.e., $\dim X > \dim Y_{G_1}$.
We shall denote by  
$F_1$ 
a general fiber of 
${\rm cont}_{G_1}$.

\medskip

{\bf Step 3}.
{\it
In this step we show that 
for all 
$1 \leq i\leq k$, 
if 
$M_i \vert_{F_1} \equiv 0$, 
then there exists a nef Cartier divisor
$N_i$ 
on 
$Y_{G_1}$
such that 
$M_i=
{\rm cont}_{G_1}^\ast N_i$.
}

Fix $1 \leq i \leq k$
and assume that 
$M_i \vert_{F_1} \equiv 0$.\footnote{We do not need to show anything for 
$i=1$, 
since by Step 1, we are assuming that 
$M_1 \sim_{Y_{G_1}} 0$.}
If 
$M_i 
\equiv_{{\rm cont}_{G_1}} 
0$
then the conclusion follows from 
Lemma~\ref{lemma.pullback}.

We now assume that 
$M_i 
\not \equiv_{{\rm cont}_{G_1}}  
0$.
By Step 0, 
$M_i$ 
is 
relatively semiample over 
$Y_{G_1}$
and its relative Iitaka fibration induces a factorization
\begin{align}
\label{factoriz.contg1}
\xymatrix{
X \ar[rr] \ar[dr]_{{\rm cont}_{G_1}}& & Y'_i\ar[dl]
\\
& Y_{G_1}.&
}
\end{align}
As 
$M_i \vert_{F_1} \equiv 0$, 
then 
$Y'_i \to Y_{G_1}$
is birational (but not the identity as 
$M_i 
\not \equiv_{{\rm cont}_{G_1}}  
0$), 
which contradicts the maximality of 
$\rho(Y_{G_1})$
among the faces of 
$\mathrm{Nef}(X)$
containing 
$M_1$
and inducing a fiber type contraction.

\medskip

We now reorder the indices 
$\{1, \dots, k\}$
as follows:
For all
$1 \leq i \leq i_0$, 
$M_i \vert_{F_1} \equiv 0$,
whereas if 
$i_0< i \leq k$,
$M_i \vert_{F_1} \not \equiv 0$.
Then, 
$i_0 \geq 1$ 
and for 
$i\leq i_0$, 
$M_i =
{\rm cont}_{G_1}^\ast N_i$, 
where 
$N_i$
is nef and Cartier on 
$Y_{G_1}$.

\medskip

{\bf Step 4}.
{\it 
In this step, we show that 
\begin{align}
\label{diseq.step3}
\sum_{i=1}^k m_i 
&
\leq
\min (
\dim F_1 + 
\sigma(M_{i_0+1}\vert_{F_1}, \dots, M_k\vert_{F_k}), 
2\dim F_1) \\
\nonumber
&
+
\min (
\dim Y_{G_1} + 
\sigma(N_{1}, \dots, N_{i_0}), 
2\dim Y_{G_1}). 
\end{align}
}

It suffices to show that both the following two inequalities hold at once
\begin{align}
\label{diseq.fiber}
   \sum_{i=i_0+1}^{k} m_i
& \leq 
\min (
\dim F_1 + 
\sigma(M_{i_0+1}\vert_{F_1}, \dots, M_k\vert_{F_k}), 
2\dim F_1), 
\\
\label{diseq.base}
    \sum_
    {j=1}^{i_0}  m_j
& \leq 
\min (
\dim Y_{G_1} + 
\sigma(N_{1}, \dots, N_{i_0}), 
2\dim Y_{G_1}).
\end{align}

First, we show that 
\eqref{diseq.base}
holds.
Since each $M_i$ is semiample, see Step 0, 
we can choose  an effective $\mathbb{R}$-divisor  $\Xi$ such that $\Xi \sim_{\mathbb R} \sum_{i=i_0+1}^k m_iM_i$ and $(X,\Delta+\Xi)$ is a klt pair.
Then, applying the canonical bundle formula for $(X,\Delta+\Xi)$ and $\mathrm{cont}_{G_1}$,
there exists an effective divisor 
$\Gamma'$
on 
$Y_{G_1}$ such that 
\begin{align}
\label{eqn.cbf.base}
0
\equiv_{{\rm cont}_{G_1}} 
K_X+\Delta+
 \sum_{i=i_0+1}^{k} m_i M_i 
 \sim_{\mathbb R} K_X+\Delta+\Xi
 \sim_\mathbb{R}
{\rm cont}^\ast_{G_1}(K_{Y_{G_1}}+\Gamma'), 
\end{align}
and
$(Y_{G_1}, \Gamma')$
is klt.
Hence, 
$K_{Y_{G_1}}+\Gamma'+\sum_{j=1}^{i_0} m_j N_j \equiv 0$ 
and 
\eqref{diseq.base}
follows by induction on the dimension once we note that 
$-K_{Y_{G_1}}$
is big, which is a direct consequence of 
Lemma~\ref{big.anticanonical.lemma}.

To show that
\eqref{diseq.fiber}
holds, 
it suffices to note that
\begin{align*}
0\equiv
(K_X+\Delta+
\sum_{i=1}^k m_iM_i ) 
\vert_{F_1} 
\equiv &
(K_X+\Delta+
\sum_{i=i_0+1}^{k} m_iM_i ) 
\vert_{F_1} 
\equiv
\\
\equiv &
K_{F_1}+\Delta\vert_{F_1}+
\sum_{i=i_0+1}^{k} m_iM_i\vert_{F_1}.
\end{align*}
The divisors 
$M_i\vert_{F_1}$
are Cartier and nef.
Moreover, 
it is immediate to verify that
$\Delta\vert_{F_1}+
\sum_{i=i_0+1}^{k} m_iM_i\vert_{F_1}
\equiv 
\Delta\vert_{F_1}+
\sum_{i=1}^{k} m_iM_i\vert_{F_1}$
is big -- simply because the restriction of a big divisor to a general fiber of a morphism is still big.
Hence, the sought inequality follows by induction on the dimension.

\medskip

\medskip

{\bf Step 5}.
{\it 
In this step, we show that
\begin{align}
\label{eqn.step4}
& \sigma(N_{1}, \dots, N_{i_0}) +
\sigma(M_{i_0+1}\vert_{F_1}, \dots, M_k\vert_{F_1})
=
\sigma(M_1, \dots, M_k), 
\ \text{and}\\
\label{diseq:chained}
\sum_{i=1}^k m_i
\leq 
&
\min (
\dim F_1 + 
\sigma(M_{i_0+1}\vert_{F_1}, \dots, M_k\vert_{F_k}), 
2\dim F_1) 
\\
\nonumber
+
&
\min (
\dim Y_{G_1} + 
\sigma(N_{1}, \dots, N_{i_0}), 
2\dim Y_{G_1})
\\
\nonumber
\leq
&
\minimo.
\end{align}
In particular, we deduce that 
inequality  
\eqref{gen.mukai.thm.eqn}
holds.
}

The inequalities 
\eqref{diseq.step3}, 
\eqref{eqn.step4} 
together imply 
\eqref{diseq:chained}, 
since 
$\dim X= \dim F_1 + \dim Y_{G_1}$.
Moreover, 
\eqref{gen.mukai.thm.eqn}
immediately follows from 
\eqref{diseq:chained}.
Hence, it suffices to prove
\eqref{eqn.step4}.

Since each 
$M_i$
descends to 
$N_i$
on 
$Y_{G_1}$, 
for $1 \leq i \leq i_0$,
then
\begin{align*}
\sigma(M_1, \dots, M_{i_0})
=
\sigma(N_1, \dots, N_{i_0}).
\end{align*}
Let us consider the morphism induced by restriction to a general fiber 
$F_1$ 
of 
${\rm cont}_{G_1}$
\begin{align*}
    r_{F_1} \colon 
    \NS(X)
    \to 
     & 
     \NS(F_1)
    \\
    [D]
    \mapsto
    &
    [D\vert_{F_1}].
\end{align*}
By construction,
$\ker r_{F_1} \supset \Sigma(M_1, \dots, M_{i_0})$.
Moreover, by the definition of 
$r_{F_1}$, 
${\rm Im}(r_{F_1}) \supset \Sigma (M_{i_0+1}\vert_{F_1}, \dots, M_{k}\vert_{F_1})$.
Thus, 
\eqref{eqn.step4} follows from 
the rank-nullity theorem 
for
$r_{F_1}\vert_{\Sigma(M_1, \dots, M_k)}$.

\medskip

From this point of the proof onwards, we shall assume that equality holds in 
\eqref{gen.mukai.thm.eqn}, 
and we shall prove that 
\ref{gen.mukai.thm.2nd.ineq}-\ref{gen.mukai.thm.cond.boundary}
in the statement of the theorem hold.

\medskip

{\bf Step 6}.
{\it In this step, we show that equality holds also in 
\eqref{diseq.step3}, and 
we illustrate the geometric consequences of this fact}.

If equality holds in 
\eqref{gen.mukai.thm.eqn}, 
then equality holds everywhere in
\eqref{diseq:chained} which implies the first part of the statement.
As equality holds in
\eqref{diseq.step3}, 
then it must hold in both
\eqref{diseq.fiber}
and 
\eqref{diseq.base}.
In such circumstance, as
$0< \dim F_1, \ \dim Y_{G_1} < \dim X$,
the inductive hypothesis implies that the following hold:
\begin{itemize}
    \item[a)]
$Y_{G_1}$
is a product of projective spaces
\begin{align*}
   Y_{G_1}
   \simeq
   \prod_{t=1}^g
   \mathbb P^{\mu_t}.
\end{align*}
We denote by 
${\rm pr}_{G_1, t}
\colon 
Y_{G_1} \to \mathbb P^{\mu_t}$
the projection onto the 
$t$-th factor;

    \item[b)]
the divisors
$N_j$, 
$j=1, \dots, i_0$,
generate 
$\NS(Y_{G_1})$,
$\Gamma' =
0$, 
and each 
$N_j$
is the pullback of a hyperplane section from one of the projective space factors of 
$Y_{G_1}$;
    \item[c)]
$F_1$ is a product of projective spaces
\begin{align*}
   F_1
   \simeq
   \prod_{s=1}^h
   \mathbb P^{\nu_s};
\end{align*}

    \item[d)]
the divisors
$M_i\vert_{F_1}$, 
$i=i_0+1, \dots, k$
generate
$\NS(F_1)$. 
Each $M_i|_{F_1}$, with $i \in \{i_0+1,\dots,k\}$
is the pullback of a hyperplane section from one of the projective space factors of 
$F_1$;
and,
    \item[e)]
$\Delta\vert_{F_1}=0$,
thus
$\Delta$
is 
${\rm cont}_{G_1}$-vertical;
moreover,
${\rm cont}_{G_1}({\rm Supp}(\Delta))$
has codimension at least 2 on $Y_{G_1}$, 
since $\Gamma'=0$;
      \item[f)]
as $\Gamma'=0$, 
$K_X+\Delta+
 \sum_{i=i_0+1}^{k} m_i M_i 
 \sim_{\mathbb R} 
{\rm cont}^\ast_{G_1}K_{Y_{G_1}}$;
        \item[g)]
the following two inequalities holds
\begin{align*}
& \dim F_1+\sigma(M_{i_0+1}\vert_{F_1}, \dots, M_{k}\vert_{F_1}) \leq 2\dim F_1, 
\quad
\text{and}, \\
& \dim Y_{G_1}+\sigma(N_{1}, \dots, N_{i_0}) \leq 2\dim Y_{G_1}    
\end{align*}
\end{itemize}

{\bf Step 7}.
{\it
In this step, we show that 
\ref{gen.mukai.thm.2nd.ineq}
holds}

This is an immediate consequence of g) above and 
\eqref{eqn.step4}.

\medskip

{\bf Step 8}
{\it
In this step, we show that 
$F_1$ is a projective space
and
\begin{align*}
\sigma(M_1, \dots, M_k)=1+\rho(Y_{G_1})=
g+1.
\end{align*}
}

As equality holds in 
\eqref{gen.mukai.thm.eqn},
and so also in \eqref{diseq.step3},  
then, by d)-e) in Steps 6, 
the numerical classes of the divisors
$M_{i_0+1}\vert_{F_1}, \dots, M_k\vert_{F_1}$
generate 
$\NS(F_1)$, 
and 
\begin{align*}
K_{F_1}+ 
\sum_{i=i_0+1}^{k} m_i M_{i}\vert_{F_1} 
\sim_{\mathbb R} 0.
\end{align*}

If 
$\rho(F_1)=1$, 
then 
$F_1 \simeq \mathbb P^\nu$, 
where 
$\nu+1 = \sum_{i=i_0+1}^{k} m_i$. 
Otherwise, 
$\rho(F_1)>1$,
and
the inductive hypothesis implies that 
each
of the $M_i\vert_{F_1}$, 
say 
$M_{i_0+1}\vert_{F_1}$,
lies in
$\partial {\rm Nef}(F_1) \setminus \{0\}$.
By Step 0,
$M_{i_0+1}$
is relatively semiample over 
$Y_{G_1}$ 
inducing a factorization as in
\eqref{factoriz.contg1}, with 
$\dim X>\dim Y'>\dim Y_{G_1}$
contradicting the maximality of 
$\rho(Y_{G_1})$ 
among fiber spaces.
Hence,
$\rho(F_1)=1$
and Step 1 implies that
$F_1 
\simeq 
\mathbb P^\nu$, 
with
$\nu=\dim X -\dim Y_{G_1}$.

Finally,
\eqref{eqn.step4}
implies that 
$\sigma(M_1, \dots, M_k)=1+\rho(Y_{G_1})$ 
since, 
by b) in Step 6
$\sigma(N_1, \dots, N_{i_0})=\rho(Y_{G_1})$`.
As 
$Y_{G_1}$
is a product of 
$g$
projective spaces, then 
$\rho(Y_{G_1})=g$.

\medskip

{\bf Step 9}.
{\it
In this step, we conclude the proof: namely, we show that
\ref{gen.mukai.thm.cond.=}-\ref{gen.mukai.thm.cond.boundary}
hold.}

From f) in Step 6, equality in 
\eqref{gen.mukai.thm.eqn} implies that
\[
K_X+\Delta+
\sum_{j=i_0+1}^k 
m_jM_j 
\equiv 
{\rm cont}^\ast_{G_1}
K_{Y_{G_1}}.
\]

As for all 
$i$, 
$M_i$
is semiample on 
$X$,
choosing 
$D_i\in \vert m_iM_i \vert_{\mathbb R}$
general, 
the log pair
$(X, \Delta + \sum_{j=i_0+1}^k D_j)$
is klt, 
and
\begin{align*}
K_X+\Delta+
\sum_{j=i_0+1}^k 
D_j 
\sim_{\mathbb R} K_{Y_{G_1}}.
\end{align*}
Thus, the boundary and moduli part in the canonical bundle formula for 
$K_X+\Delta+
\sum_{j=i_0+1}^k 
D_j $ 
over
$Y_{G_1}$ 
are (numerically) trivial. 
Hence, by standard results on the canonical bundle formula, see
\cite[Theorem 2.22]{DcS}, 
there exists a finite Galois cover
$\tau \colon \tilde Y \to Y_{G_1}$ which is 
\'etale in codimension one,
an open set 
$U \subset \tilde Y$,
and an isomorphism 
\begin{align}
\label{isom:CBF.isotrivial}
\left (
X, \Delta +
\sum_{j=i_0+1}^k D_j
\right ) 
\times_{Y_{G_1}}
U 
\simeq
\left (
\mathbb P^\nu, 
\sum_{j=i_0+1}^k E_j
\right ) 
\times
U
\end{align}
over $U$, 
where the 
$E_j$
are suitable 
$\mathbb R$-divisors on 
$\mathbb P^\nu$
and the pair
$(\mathbb P^\nu, \sum_{j=i_0+1}^k E_j)$
is klt.
As
$Y_{G_1}$
is smooth and simply connected by the inductive hypothesis,
then 
$\tau$
is an isomorphism.
Moreover, 
by \cite[Proposition 4.4]{Ambro.mod},
we can take 
$U$ 
to be a big open set of 
$Y_{G_1}$.
Hence there exists a birational map 
\[
\Phi \colon X \dashrightarrow 
\mathbb P^\nu \times Y_{G_1}.
\]

Let us notice that $\Phi$ is a birational contraction:
indeed, as $U$ is a big open set, $\Phi^{-1}$ cannot contract any divisor, 
as it is an isomorphism over the big open set 
$\mathbb P^\nu \times U \subset \mathbb P^\nu \times Y_{G_1}$.
Since
\[
0
\sim_\mathbb{R}
\Phi_\ast (K_X+\Delta+\sum_{i=1}^k m_iM_i)
\sim_\mathbb{R}
K_{\mathbb P^\nu \times Y_{G_1}}+ \sum_{i=1}^{i_0} m_i \mathrm{pr}_{\mathbb Y_{G_1}}^\ast N_i+
\sum_{j=i_0+1}^{k} m_j \mathrm{pr}_{\mathbb P^\nu}^\ast E_j
\]
then,
by \eqref{isom:CBF.isotrivial},
$\Phi$
is crepant with for the  generalized pairs
\[
(X, \Delta+\sum_{i=1}^k m_iM_i), 
(\mathbb P^\nu \times Y_{G_1}, 0+ \sum_{i=1}^{i_0} m_i \mathrm{pr}_{\mathbb Y_{G_1}}^\ast N_i+
\sum_{j=i_0+1}^{k} m_j \mathrm{pr}_{\mathbb P^\nu}^\ast E_j),
\] 
where the moduli part of the latter is the Cartier closure of 
\[\sum_{i=1}^{i_0} m_i \mathrm{pr}_{\mathbb Y_{G_1}}^\ast N_i+
\sum_{j=i_0+1}^{k} m_j \mathrm{pr}_{\mathbb P^\nu}^\ast E_j.
\]

Since 
$(\mathbb P^\nu \times Y_{G_1}, 0+ \sum_{i=1}^{i_0} m_i \mathrm{pr}_{\mathbb Y_{G_1}}^\ast N_i+
\sum_{j=i_0+1}^{k} m_j \mathrm{pr}_{\mathbb P^\nu}^\ast E_j)$
has generalized terminal singularities, then 
$\Phi$
cannot contract any divisors, as the exceptional divisors all have generalized log discrepancy $\leq 1$ for 
$(\mathbb P^\nu \times Y_{G_1}, 0+ \sum_{i=1}^{i_0} m_i \mathrm{pr}_{\mathbb Y_{G_1}}^\ast N_i+
\sum_{j=i_0+1}^{k} m_j \mathrm{pr}_{\mathbb P^\nu}^\ast E_j)$.
Thus, 
$\Phi$ is an isomorphism in codimension one.

As 
$\mathbb P^\nu \times Y_{G_1}$
is isomorphic to the product of projective spaces
$\mathbb P^\nu \times \prod_{t=1}^g
\mathbb P^{\mu_t}$, 
it
does not admit any small 
$\mathbb Q$-factorial 
modification, since
$    {\rm Nef}
    \left (
    \mathbb P^\nu
    \times
    \prod_{t=1}^g
\mathbb P^{\mu_t}
    \right )
    = {\rm Pseff}
    \left (
    \mathbb P^\nu
    \times
    \prod_{t=1}^g
\mathbb P^{\mu_t}
    \right )$.
Hence,
    $\Phi$
can be extended to an isomorphism 
$    \Phi 
    \colon 
    X 
    \to 
    \mathbb P^\nu 
    \times 
    Y_{G_1}$
    showing that
\ref{gen.mukai.thm.cond.=}
holds.

By
\eqref{isom:CBF.isotrivial},
$\Phi_\ast D_j= E_j$
for any 
$j=i_0+1, \dots, k$ 
and, thus,
$m_j M_j \sim_\mathbb R \Phi_1^\ast L_j$, 
where 
$\Phi_1$ denotes the composition of 
$\Phi$ 
with the projection of 
$\mathbb P^\nu 
\times 
Y_{G_1}$
onto the first factor and 
$L_j$ 
is an ample 
$\mathbb R$-divisor.
Then 
\eqref{diseq.base}
implies that 
for all 
$j=i_0+1, \dots, k$,
$L_j$ 
is a hyperplane section, which implies
\ref{gen.mukai.thm.cond.boundary}.

To see that also 
\ref{gen.mukai.thm.cond.delta.null}
holds, 
it suffices to notice that, by our construction, 
\begin{align*}
    K_X+\sum_{i=1}^k m_iM_i 
    \sim 
    \Phi^\ast 
    \left (
    K_{\mathbb P^\nu \times Y_{G_1}}
    +
    \sum_{i=1}^{i_0} 
    m_i 
    ({\rm pr}_{Y_{G_1}}^\ast N_i)
    +
    \sum_{j=i_0+1}^k
    m_j
    ({\rm pr}_{\mathbb P^\nu}^\ast L_j)
    \right )
    \sim_\mathbb{R} 0,
\end{align*}
where 
$ {\rm pr}_{Y_{G_1}} $, 
$ {\rm pr}_{\mathbb P^\nu} $
are the projections onto the two factors.
As also 
$ K_X+\Delta+\sum_{i=1}^k m_iM_i
\sim_\mathbb{R} 0$,
by hypothesis, 
then 
$\Delta =0$.
\end{proof}

We now conclude this section by proving the results stated in the introduction that directly depend upon those we have just proven.

\begin{proof}[Proof of Theorem~\ref{introthm:charct-p1n}]
By considering $(X,\sum_i a_iM_i)$ as a generalized klt pair, 
the stated result now follows from Theorem~\ref{gen.mukai.thm}, condition \ref{gen.mukai.thm.2nd.ineq} in its statement, cf. Remark~\ref{rem.eq.gen.mukai.thm}.
\end{proof}

\begin{proof}[Proof of Corollary~\ref{introcor:poly}]
The integral reflexive polytope $P\subset \mathbb{Q}^n$ corresponds to a klt Fano toric variety $T(P)$ (see, e.g.,~\cite[Theorem 8.3.4]{CLS11}).
The Minkowski sum decomposition
\[
P=\lambda_1P_1 \oplus \dots \oplus \lambda_r P_r, 
\]
leads to a decomposition of the anti-canonical divisor 
\[
-K_{T(P)} \sim \sum_{i=1}^r \lambda_i M(P_i)
\]
where each $M(P_i)$ is a nef Cartier divisor on $T(P)$ (see~\cite[Remark 2.6]{Bor93}).
By Theorem~\ref{gen.mukai.thm}, we know that
\begin{equation}\label{toric:nef-comp} 
\sum_{i=1}^r \lambda_i \leq \dim {\rm Cl}_{\mathbb{Q}}(T(P))+\dim(T(P)).
\end{equation} 
The dimension of the $\mathbb{Q}$-Class group ${\rm Cl}_{\mathbb{Q}}(T(P))$ is precisely the number 
of facets of $P$, denoted by $f$, minus $\dim(P)$ (see, e.g.,~\cite[Theorem 4.13]{CLS11}).
Thus, we conclude that 
\begin{equation}\label{ineq-facets} 
\sum_{i=1}^r \lambda_i \leq f. 
\end{equation} 
Furthermore, the previous equality holds if and only if the equality in~\eqref{toric:nef-comp} holds.
Thus, if~\eqref{ineq-facets} holds, then $T(P)$ must be isomorphic to a product of projective spaces. So, $P$ must be a product of simplices.
\end{proof}

\section{The total index of smooth del Pezzo surfaces}
\label{sec:total-index-dP}

In this section, we compute the total indices of smooth del Pezzo surfaces: $X=\mathbb{P}^2$, $X=\mathbb{P}^1\times\mathbb{P}^1$, and the blow-up
$X=\mathrm{Bl}_r(\mathbb{P}^2)$ of $r$ points in general position for $1\le r\le 8$. The degree of $X$ is $d \coloneqq K_X^2=(-K_X)^2$, which equals $9-r$ for $X=\mathrm{Bl}_r(\mathbb{P}^2)$. 
We will consider, simultaneously, the total index over the integers $\tau_X(\mathbb{Z})$ and the total index over the rationals $\tau_X(\mathbb{Q})$ as defined below. 

\begin{definition}
Let $R$ be either $\mathbb{Q}$ or $\mathbb{Z}$. 
The {\em integral (resp. rational) total index} of a smooth Fano variety $X$ is defined to be
\[
\tau_X(R) \coloneqq \max\left\{\,
\sum_{i=1}^k a_i \,\middle|\,
-K_X \equiv \sum_{i=1}^k a_iL_i, \forall i,\text{ $a_i\in R$, $L_i$ is nef Cartier
with $L_i\not\equiv 0$}
\right\} .
\]
\end{definition}

\begin{remark}
If $X$ is a $\mathbb Q$-Fano variety, the anticanonical divisor $-K_X$ is in general only $\mathbb{Q}$-Cartier. 
Since the decomposition 
\[
  -K_X \equiv \sum a_i L_i
\]
with integer coefficients forces $-K_X$ to be Cartier, the integer–coefficient total index $\tau_X(\mathbb{Z})$ 
is not well-defined unless $X$ is smooth (or at least Gorenstein). 
In such singular cases, one can only define a rational–coefficient version $\tau_X(\mathbb{Q})$ using $\mathbb{Q}$-Cartier divisors.
\end{remark}

Let $\pi:X\to\mathbb{P}^2$ be the blow-up of $r$ points $p_1, \dots, p_r$ in general position with exceptional curves $E_1 \coloneqq \pi^{-1}(p_1),\dots,E_r \coloneqq \pi^{-1}(p_r)$.
Set $H \coloneqq \pi^\ast \mathcal{O}_{\mathbb{P}^2}(1)$. In $\mathrm{N}^1(X)$ we write a Cartier divisor $L$ as
\[
L=xH-\sum_{i=1}^r y_i E_i,\qquad x,y_i\in\mathbb{Z},
\]
with intersection pairing
\begin{equation}\label{eq:inter}
 \left(xH-\sum y_iE_i \right)\cdot \left(x'H-\sum y'_iE_i \right)=xx'-\sum_i y_i y'_i.
\end{equation}
The anticanonical divisor is $-K_X \sim 3H-\sum_{i=1}^r E_i$.

We will use the following classes of irreducible curves:
\begin{itemize}[leftmargin=2em]
\item $E_i$ the exceptional curves, with $E_i^2=-1$;
\item $H-E_i-E_j$ for $i\ne j$, the class of the strict transform of the line through $p_i$ and $p_j$,
which is a (-1)-curve;
\item for any $4$-subset $I\subset\{1,\dots,r\}$ when $r\ge 4$, define
\[
 f_{[I]} \coloneqq 2H-\sum_{i\in I}E_i,
\]
which is the strict transform of a conic through the four points in $I$. The class $f_{[I]}$ is basepoint-free and has self-intersection $0$;
\item for any $5$-subset $I\subset\{1,\dots,r\}$ when $r\geq 5$, define
\[
 g_{[I]} \coloneqq 2H-\sum_{i\in I}E_i,
\]
which is the strict transform of the conic through the five points in $I$. This $g_{[I]}$ is a $(-1)$-curve;
\item for any $6$-subset $J=\{1,\dots,7\}\setminus\{i\}$,
define
\[
 h_{[J]} \coloneqq 3H-2E_i-\sum_{j\in J}E_j,
\]
which is the strict transform of a plane cubic curve having a double point at $p_i$ and passing simply through each of the other six blown-up points. This $h_{[J]}$ is a $(-1)$-curve.
\end{itemize}

\begin{lemma}\label{lem:nef-ineq}
Let $X=\mathrm{Bl}_r(\mathbb{P}^2)$. If $L=xH-\sum_{i=1}^r y_iE_i$ is a nef Cartier divisor on $X$, then
the following inequalities hold:
\begin{enumerate}[label=\textup{(\alph*)}]
\item $y_i\ge 0$ for all $i$; \hfill ($L \cdot E_i$)
\item $x-y_i-y_j\ge 0$ for all distinct $i,j$; \hfill ($L \cdot (H-E_i-E_j)$)
\item if $r\ge5$, then $2x-\sum_{i\in I}y_i\ge 0$ for all $5$-subsets $I\subset\{1,\dots,r\}$; \hfill ($L \cdot g_{[I]}$)
\item if $r=7$, then $3x-2y_i-\sum_{j\in J}y_j\ge 0$ for all $6$-subsets $J=\{1,\dots,7\}\setminus \{i\}$. \hfill ($L \cdot h_{[J]}$)
\end{enumerate}
\end{lemma}

\begin{proof}
A nef Cartier divisor has nonnegative intersection with every effective curve. Each item follows by intersecting $L$ with the corresponding irreducible curves listed above.
\end{proof}

Define
\begin{equation}\label{eq:mX}
 m(X) \coloneqq \min\{\,(-K_X\!\cdot\!L)\mid L\ \text{nef Cartier},\ L\not\equiv0\,\}.
\end{equation}

\begin{proposition}[Base certificate]\label{prop:base}
The function
\begin{equation}\label{eq:Phi}
 \Phi_{\mathrm{base}}(L) \coloneqq \frac{1}{m(X)}\,(-K_X\!\cdot\!L)
\end{equation}
satisfies $\Phi_{\mathrm{base}}(L)\ge 1$ for all nef Cartier $L\not\equiv0$.
Consequently,
\begin{equation}\label{eq:bounds}
 \tau_X(\mathbb{Q})\le \frac{d}{m(X)}\,,\qquad \tau_X(\mathbb{Z})\le \left\lfloor\frac{d}{m(X)}\right\rfloor,
\end{equation}
where $d=(-K_X)^2$.
\end{proposition}

\begin{proof}
By definition of $m(X)$, we have $(-K_X\!\cdot\!L)\ge m(X)$
for every nef Cartier divisor $L\not \equiv 0$.
If $-K_X\equiv\sum a_iL_i$, then
\(
\sum a_i\le \sum a_i\,\Phi_{\mathrm{base}}(L_i)=\Phi_{\mathrm{base}}(-K_X)=d/m(X).
\)
In the integer-coefficient case, $\sum a_i\in\mathbb{Z}$, hence our assertion holds.
\end{proof}

\begin{remark} In Proposition~\ref{prop:base}, the term \emph{certificate} means a numerical witness that verifies or realizes the upper bound of the total index~$\tau_X$. 
The \emph{base certificate} is the simplest and most fundamental form of such a witness, given by
\[
\Phi_{\mathrm{base}}(L)=\frac{1}{m(X)}(-K_X\cdot L),
\]
which provides the basic numerical data for determining the maximal value of~$\tau_X$.

\end{remark}

\begin{proposition}\label{prop:m-values}
For smooth del Pezzo surfaces $X$, we have
\begin{equation}\label{eq:m-values}
 m(X)=
 \begin{cases}
 3,& X=\mathbb{P}^2,\\
 2,& X=\mathbb{P}^1\times\mathbb{P}^1,\\
 2,& X=\mathrm{Bl}_r(\mathbb{P}^2)\ \text{for }1\le r\le 7,\\
 1,& X=\mathrm{Bl}_8(\mathbb{P}^2).
 \end{cases}
\end{equation}
\end{proposition}

\begin{proof}
If $X=\mathbb{P}^2$ or $\mathbb{P}^1\times\mathbb{P}^1$, our assertion is trivial. Thus we assume $X=\mathrm{Bl}_r(\mathbb{P}^2)$ ($1\le r \le 8$). Let $L=xH-\sum_{i=1}^r y_iE_i$ with $x,y_i\in\mathbb{Z}$ be a numerically nontrivial nef Cartier divisor on $X$. Intersecting $L$ with $H$ and $E_i$, we have $x, y_i \ge 0$.
We also see that $x\ge1$; otherwise (b) of Lemma~\ref{lem:nef-ineq} forces $L\equiv0$. We now prove $(-K_X\cdot L)=3x-\sum_{i=1}^r y_i\ge 2$ in all cases $1\le r\le7$:

\smallskip
\noindent\underline{The case $x=1$}: By Lemma~\ref{lem:nef-ineq} (b), $y_i+y_j\le1$ for all $i\ne j$. With $y_i\in\mathbb{Z}_{\ge0}$ this implies $\sum_{i=1}^r y_i\le1$, so $3x-\sum_{i=1}^r y_i\ge 2$.

\smallskip
\noindent\underline{The case $x\ge2$, $r\le4$}: Summing Lemma~\ref{lem:nef-ineq} (b) over pairs yields
$\sum_{i=1}^r y_i\le 2x$, hence $3x-\sum_{i=1}^r y_i\ge x\ge2$.

\smallskip
\noindent\underline{The case $x\ge2$, $r=5$}: Lemma~\ref{lem:nef-ineq} (c) with $I=\{1,\dots,5\}$ gives $\sum_{i=1}^5 y_i\le 2x$, thus $3x-\sum_{i=1}^5 y_i\ge x\ge2$.

\smallskip
\noindent\underline{The case $x\ge2$, $r=6$}: For each $j$, apply Lemma~\ref{lem:nef-ineq} (c) to $I=\{1,\dots,6\}\setminus\{j\}$ and sum over $j$ to obtain $12x-5S\ge0$ with $S=\sum_{i=1}^6 y_i$. This yields that $\dfrac{12}{5}x\ge S$. Then we obtain $3x-S\ge 3x-\dfrac{12}{5}x=\dfrac{3}{5}x\ge \dfrac{6}{5}$ for all $x\ge2$. Since $3x-S$ is an integer, $3x-S\ge 2$.

\smallskip
\noindent\underline{The case $x\ge2$, $r=7$}: Using Lemma~\ref{lem:nef-ineq} (d), we obtain $S+y_i\le 3x$ for each $i$, where $S \coloneqq \sum_{i=1}^7 y_i$. Summing over $i$ yields
\begin{equation}\label{eq:delta-sum}
 8S \le 21x \quad\Longrightarrow\quad S \le \left\lfloor\frac{21}{8}x\right\rfloor.
\end{equation}
If $x\ge3$ then $\lfloor 21x/8\rfloor \le 3x-2$, hence $3x-S\ge2$.

For $x=2$ the estimate \eqref{eq:delta-sum} gives $S\le5$. We claim that $S\le 4$. To prove this, assume the contrary, i.e., $S=5$. 
The inequalities $x-y_i-y_j\ge 0$ force $y_i+y_j\le 2$ for all $i\ne j$.
Hence, at most one of the $y_i$ can take the value $2$: if $y_k=2$ then
all other $y_j$ must vanish, so the total sum is $S=\sum_{i=1}^7 y_i=2$. This is a contradiction. 
Thus, necessarily all $y_i\in\{0,1\}$ and at least
five of them equal $1$. But for the corresponding $5$--subset $I$ one has
\(
2x-\sum_{i\in I} y_i = 4-\sum_{i\in I}y_i \ge 0
\)
by Lemma~\ref{lem:nef-ineq} (c), whereas $\sum_{i\in I}y_i\ge 5$ under the assumption.
This is a contradiction. Therefore $S\le 4$, and consequently
$3x-S \ge 6-4=2$ as required.

In all cases $(-K_X\cdot L)\ge2$. The bound is attained by the nef Cartier divisor $H-E_1$, since $(-K_X\cdot(H-E_1))=2$ for every $1\le r\le7$. For $r=8$, the anticanonical divisor is nef with $(-K_X)^2=1$, so $m(X)=1$ attained by $-K_X$.
\end{proof}

Combining Propositions~\ref{prop:base} and \ref{prop:m-values} yields universal bounds
\begin{equation}\label{eq:univ}
 \tau_X(\mathbb{Q})\le \frac{d}{m(X)},\qquad \tau_X(\mathbb{Z})\le \left\lfloor\frac{d}{m(X)}\right\rfloor,
\end{equation}
which are sharp in all cases except $X=\mathrm{Bl}_r(\mathbb{P}^2)$ with $r=1,2$. In those two cases the improved certificate $\Phi_H(L) \coloneqq L\cdot H$ satisfies $\Phi_H(L)\in\mathbb{Z}_{\ge1}$ for all nef Cartier $L\not\equiv0$ and yields $\tau_X(R)\le (-K_X\cdot H)=3$.

\begin{theorem}
\label{thm:main.dP}
For a smooth del Pezzo surface $X$, one has
\begin{equation}\label{eq:tau-values}
 \tau_X(\mathbb{Q})=
 \begin{cases}
 3,& X=\mathbb{P}^2,\\
 4,& X=\mathbb{P}^1\times\mathbb{P}^1,\\
 3,& X=\mathrm{Bl}_r(\mathbb{P}^2),\ r=1,2,3,\\
 \dfrac{9-r}{2},& X=\mathrm{Bl}_r(\mathbb{P}^2),\ r=4,5,6,7,\\
 1,& X=\mathrm{Bl}_8(\mathbb{P}^2),
 \end{cases}
\qquad
 \tau_X(\mathbb{Z})=
 \begin{cases}
 3,& X=\mathbb{P}^2,\\
 4,& X=\mathbb{P}^1\times\mathbb{P}^1,\\
 3,& X=\mathrm{Bl}_r(\mathbb{P}^2),\ r=1,2,3,\\
 2,& X=\mathrm{Bl}_r(\mathbb{P}^2),\ r=4,5,\\
 1,& X=\mathrm{Bl}_r(\mathbb{P}^2),\ r=6,7,8.
 \end{cases}
\end{equation}
\end{theorem}

\begin{proof}
First, we prove the upper bound.
Apply \eqref{eq:univ} with $m(X)$ from Proposition~\ref{prop:m-values}; when $r=1,2$ replace the base certificate by $\Phi_H$ to obtain the exact bound $3$.

Now, we turn to prove the lower bound and thus optimality.
We present explicit nef decompositions of $-K_X$ attaining the values in \eqref{eq:tau-values}.

\smallskip\noindent
\textbf{$X=\mathbb{P}^2$:} $-K_X=H+H+H$.

\smallskip\noindent
\textbf{$X=\mathbb{P}^1\times\mathbb{P}^1$:} Writing $H_1,H_2$ for the rulings, $-K_X=H_1+H_1+H_2+H_2$.

\smallskip\noindent
\textbf{$X=\mathrm{Bl}_r(\mathbb{P}^2)$:} With our notation,
\begin{itemize}[leftmargin=2em]
\item $r=1$: $-K_X=(H-E_1)+H+H$;
\item $r=2$: $-K_X=(H-E_1)+(H-E_2)+H$;
\item $r=3$: $-K_X=(H-E_1)+(H-E_2)+(H-E_3)$;
\item $r=4$: for the rational decomposition we can consider $-K_X=\tfrac12 f_{[1234]}+\sum_{i=1}^4 \tfrac12(H-E_i)$ while we can consider $-K_X=f_{[1234]}+H$ for the integral decomposition;
\item $r=5$: $-K_X=f_{[1234]}+(H-E_5)$;
\item $r=6$: for the rational decomposition we consider $-K_X=\tfrac12 f_{[1234]}+\tfrac12 f_{[3456]}+\tfrac12 f_{[1256]}$; for the integer decomposition, the one-term presentation by $-K_X$ gives $1$;
\item $r=7$: $-K_X$ is nef; a one-term presentation gives $1$;
\item $r=8$: likewise, $-K_X$ is nef with $(-K_X)^2=1$; one term.
\end{itemize}
These reach the claimed totals, proving optimality together with the matching upper bounds.
\end{proof}

\begin{remark}\label{Ito}We can find examples such that $\tau_X(\mathbb{Q}) \not=\tau_X(\mathbb{Z})$ in the above list. Note that  the first calculation on such an example is done by Atsushi Ito (cf. \cite[Example 0.11]{gon-kinosaki}).
\end{remark}

\section{Smooth Fano threefolds of nef complexity one}
\label{sec:smooth-Fano-3folds}

Our main Theorem~\ref{mukai type gklt} characterizes products of projective spaces as the varieties with nef complexity $c_X$ equal to zero.
We actually expect that for a klt Fano variety $X$ we have $c_X=0$ whenever $c_X<1$,
however, our proof does not give this statement. 
In the case of smooth Fano varieties, we can only produce a few examples for which $c_X=1$. This motivates the following question.

\begin{question}
Are there smooth Fano varieties $X$ with $0<c_X \le 1$ other than the following four cases?
\begin{enumerate}
\item The $n$-dimensional smooth quadric $Q^n \subset \mathbb{P}^{n+1}$.
\item The blow-up of $\mathbb{P}^n$ along a linear subspace; namely,
\[
\mathbb{P}\bigl(\mathcal{O}_{\mathbb{P}^{n-\ell-1}}(1) \oplus
\mathcal{O}_{\mathbb{P}^{n-\ell-1}}^{\oplus (\ell + 1)}\bigr).
\]
\item A smooth hyperplane section of the Segre variety $\mathbb{P}^{n_1} \times \mathbb{P}^{n_2}$.
\item A product $\bigl(\prod_j \mathbb{P}^{n_j}\bigr) \times Y$, where $Y$ is as in \textup{(1)}–\textup{(3)}.
\end{enumerate}
\end{question}

In this section, we address this question for smooth Fano threefolds. From now on, we denote by $\tau_X$ the total index of $X$ over the rational numbers. The argument below works verbatim for the integral total index.

\begin{proposition}\label{prop:tau-length:ineq}
Let $X$ be a smooth Fano variety. Let $B$ be any set of $\rho(X)$ linearly independent extremal rays of $X$. For any extremal ray $R\subset \overline{\rm NE}(X)$, denote the length of $R$ by 
\[
\ell(R) \coloneqq \min\{-K_X\cdot C\mid C~\text{is a rational curve and~} [C]\in R\}.
\]
Then
\[
\tau_X \le \sum_{R\in B} \ell(R).
\]
\end{proposition}

\begin{proof}
Suppose $-K_X\equiv \sum a_i L_i$ is a nef decomposition realizing $\tau_X=\sum a_i$. 
Since each $L_i$ is numerically nontrivial by definition, for each $i$ there exists an extremal ray $R\in B$ such that $L_i$ has a positive intersection with the minimal curve of $R$. 
Therefore, the sum of the coefficients cannot exceed the sum of the lengths.
\end{proof}

\begin{proposition}\label{prop:blowup-decomposition}
Let $\varphi:X=\operatorname{Bl}_p(Y)\to Y$ be the blow-up of a smooth projective variety $Y$ at a point $p$, with an exceptional divisor $E\simeq \mathbb{P}^{n-1}$. Let $\ell\subset E$ be a line contained in $E$. Then every divisor $L\in \mathrm{Pic}(X)$ can be written uniquely as
\[
L \sim \varphi^{\ast}M - (L\cdot \ell)\,E,
\]
for some divisor $M$ on $Y$. Moreover, if $L$ is numerically nontrivial nef, then so is $M$.
\end{proposition}

\begin{proof}
Any divisor on $X$ is uniquely decomposed as $\varphi^{\ast}M - \beta E$ for some $\beta\in \mathbb{Z}$. Since $E\cdot \ell=-1$, one has $\beta=L\cdot \ell$. Now suppose $L$ is nef. For any curve $\Gamma\subset Y$ with multiplicity $m=\mathrm{mult}_p(\Gamma)$, consider the effective $1$-cycle $C=\widetilde\Gamma+m\ell$, where $\widetilde\Gamma$ is the strict transform of $\Gamma$. Then $E\cdot C=0$, and hence
\[
0\le L\cdot C = (\varphi^{\ast}M-(L\cdot \ell)E)\cdot (\widetilde\Gamma+m\ell)
= M\cdot \Gamma.
\]
Thus $M$ is nef. Moreover, if $M$ were numerically trivial, then $L\equiv - (L\cdot \ell)\,E$, contradicting $L\not\equiv 0$.
\end{proof}

\begin{proposition}\label{prop:blowup-fano}
Let $\varphi:X=\operatorname{Bl}_p(Y)\to Y$ be the blow-up of a smooth threefold $Y$ at a point. If $X$ is a smooth Fano threefold, then so is $Y$, and moreover
\[
\tau_X \le \tau_Y \quad \text{and} \quad c_X\ge c_Y+1.
\]
\end{proposition}

\begin{proof}
The former assertion follows from \cite[Proposition~3.4]{w-contr}. Suppose $-K_X\equiv \sum a_i L_i$ is a nef decomposition realizing $\tau_X=\sum a_i$. By Proposition~\ref{prop:blowup-decomposition}, we can write $L_i=\varphi^{\ast}M_i-(L_i\cdot \ell)E$ with $M_i$ numerically nontrivial and nef on $Y$. In summary,
\[
-K_Y = \sum a_i M_i.
\]
Thus $-K_Y$ admits a nef decomposition with total weight $\sum a_i=\tau_X$. It follows that $\tau_Y\ge \tau_X$. Moreover,
\[
 c_X=3+\rho(X)-\tau_X\ge 3+(\rho(Y)+1)-\tau_Y=c_Y+1.
\]
\end{proof}

From now on, we assume that $X$ is a smooth Fano threefold with $c_X\leq 1$. Then, by definition,
\[
\tau_X \geq \rho(X)+2.
\]

\begin{proposition}
If $X$ is a smooth Fano threefold with $\rho(X)=1$ and $c_X\leq 1$, then $X$ is isomorphic to either $\mathbb{P}^3$ or a smooth quadric threefold $Q^3\subset \mathbb{P}^4$.
\end{proposition}

\begin{proof}
If $\rho(X)=1$, then $\tau_X\geq 3$. Since $\tau_X$ coincides with the Fano index in this case, the index of $X$ is at least three. Thus, we obtain our assertion by work of Kobayashi--Ochiai \cite{kooc}.
\end{proof}

\begin{proposition}\label{prop:length3}
Suppose $X$ is a smooth Fano threefold with $c_X\leq 1$ and $\rho(X)>1$. If $X$ has an extremal ray of length at least three, then either 
\[
X\simeq \mathbb{P}^1 \times \mathbb{P}^2 \qquad 
\text{ or }
\qquad 
X \simeq \mathbb{P}\bigl(\mathcal{O}_{\mathbb{P}^1}(1)\oplus \mathcal{O}_{\mathbb{P}^1}^{\oplus 2}\bigr).
\]
\end{proposition}

\begin{proof}
The classification of extremal contractions of smooth threefolds shows that the only possibility for an extremal ray of length at least three is a $\mathbb{P}^2$-bundle over $\mathbb{P}^1$ (see \cite{Mori} or \cite[Theorem~1.4.3]{IP}). Hence $X\simeq \mathbb{P}(E)$ where $E\simeq \mathcal{O}_{\mathbb{P}^1}(a)\oplus \mathcal{O}_{\mathbb{P}^1}(b)\oplus \mathcal{O}_{\mathbb{P}^1}$ with $a\ge b\ge 0$. Since $X$ is Fano, the canonical bundle formula yields $2> a+b\ge 0$. If $a=b=0$, then $X\simeq \mathbb{P}^1\times \mathbb{P}^2$, which gives $c_X=0$. If $(a,b)=(1,0)$, then $X$ is isomorphic to $\mathbb{P}\big(\mathcal{O}_{\mathbb{P}^1}(1)\oplus \mathcal{O}_{\mathbb{P}^1}^{\oplus 2}\big)$. Denote by $\xi$ the tautological divisor on $\mathbb{P}\big(\mathcal{O}_{\mathbb{P}^1}(1)\oplus \mathcal{O}_{\mathbb{P}^1}^{\oplus 2}\big)$.
Then we have
\[
 -K_X = 3\xi + H,
\]
where $H$ is the pullback of a point in $\mathbb{P}^1$
by the natural projection $X \to \mathbb{P}^1$.
This shows that $\tau_X \ge 4$.

On the other hand, $X$ is the blow-up of $\mathbb{P}^3$
along a line; let $\pi\colon X \to \mathbb{P}^3$ denote the blow-up morphism.
Since $\pi_{\ast}(-K_X) = -K_{\mathbb{P}^3}$, we also have $\tau_X \le 4$.
As a consequence, we obtain
$\tau_X=4$ and $\rho(X)=2$, so that $c_X=1$.
\end{proof}

From now on, we assume that all extremal rays have length one or two.

\begin{proposition}\label{prop:two-extremal-length2}
Let $X$ be a smooth Fano threefold with $c_X\leq 1$. If all extremal rays of $X$ have length one or two, then either $X\simeq (\mathbb{P}^1)^3$ or $X$ has exactly two extremal rays of length two.
\end{proposition}

\begin{proof}
By a theorem of Fujita \cite[Theorem 1.4(1)]{Fuji14c}, either $X\simeq (\mathbb{P}^1)^3$ or $X$ has at most two extremal rays of length two. Our argument holds in the first case, so we consider the second. In this case, by Proposition~\ref{prop:tau-length:ineq}, we have 
\begin{equation}\label{double_ineq}
\rho(X)+2 \leq \tau_X\le \sum \ell(R)\le \rho(X)+2.
\end{equation}
This implies that there must be exactly two rays with $\ell(R)=2$, and all other rays have $\ell(R)=1$.
\end{proof}

The classification of contractions of smooth threefolds (\cite{Mori}, \cite[Theorem~1.4.3]{IP}) shows that a contraction of length two must be one of the following types:
\begin{itemize}
\item[(E2)] the blow-up of a smooth variety at a point (exceptional divisor $\mathbb{P}^2$ with normal bundle $\mathcal{O}_{\mathbb{P}^2}(-1)$);
\item[(C2)] a $\mathbb{P}^1$-bundle;
\item[(D2)] a $\mathbb{P}^1\times\mathbb{P}^1$-fibration.
\end{itemize}

\begin{lemma}\label{lem:noD2}
Let $X$ be a smooth Fano threefold with $c_X\leq 1$ and exactly two extremal rays of length two. Then $X$ does not admit any type (D2) contraction.
\end{lemma}

\begin{proof}
Suppose a contraction of an extremal ray $\varphi:X\to \mathbb{P}^1$ is a $\mathbb{P}^1\times\mathbb{P}^1$-fibration. Then $\rho(X)=2$ by \cite[Corollary 3.17]{KM98}. 
We have another elementary contraction $\psi: X\to Z$ which is also of type (E2), (C2), or (D2). Since fibers of $\varphi$ and $\psi$ intersect in at most finitely many points, $\psi$ must be a $\mathbb{P}^1$-bundle. Since $Z$ is a smooth Fano surface with $\rho(Z)=1$, it is $\mathbb{P}^2$.
In this case, the nef decomposition of $-K_X$ is given by $-K_X=a\,\varphi^{\ast}\mathcal{O}_{\mathbb P^1}(1)+b\,\psi^{\ast}\mathcal{O}_{\mathbb P^2}(1)$, where $a+b=4$ and $a,b\ge 0$;
see \eqref{double_ineq}.
For minimal rational curves $C$ and $C'$ associated with $\varphi$ and $\psi$, respectively, one has
\[
2 = -K_X \cdot C = b\,\mathcal{O}_{\mathbb{P}^2}(1)\cdot \psi_{\ast}(C)
\quad \text{and}\quad 
2 = -K_X \cdot C' = a\,\mathcal{O}_{\mathbb{P}^1}(1)\cdot \varphi_{\ast}(C').
\]
This implies 
\[
4=a+b=\frac{2}{\mathcal{O}_{\mathbb{P}^1}(1)\cdot \varphi_{\ast}(C')}
+\frac{2}{\mathcal{O}_{\mathbb{P}^2}(1)\cdot \psi_{\ast}(C)}\le 4.
\]
Therefore we obtain 
\[
a=b=2 \quad \text{and}\quad 
\mathcal{O}_{\mathbb{P}^2}(1)\cdot \psi_{\ast}(C)=
\mathcal{O}_{\mathbb{P}^1}(1)\cdot \varphi_{\ast}(C')=1.
\]
In particular, $\varphi|_{C'}: C'\to \mathbb{P}^1$ is an isomorphism. This implies that $(\varphi, \psi): X\to \mathbb{P}^1\times \mathbb{P}^2$ is an isomorphism. However, $\mathbb{P}^1\times \mathbb{P}^2$ admits an extremal ray of length three. This is a contradiction.
\end{proof}

\begin{proposition}\label{prop:two:rays}
Let $X$ be a smooth Fano threefold with $c_X\leq 1$. Assume $X$ has exactly two extremal rays of length two. Then at most one of them can be of type (E2).
\end{proposition}

\begin{proof}
Suppose that both rays of length two are of type (E2). Then the corresponding exceptional divisors are disjoint. Hence $X\simeq \operatorname{Bl}_{p,q}(Y)$, the blow-up of a smooth threefold $Y$ at two distinct points $p,q$. By Proposition~\ref{prop:blowup-fano}, this contradicts the assumption $c_X\leq 1$. Therefore there cannot be two rays of type (E2).
\end{proof}

\begin{proposition}
Let $X$ be a smooth Fano threefold with $c_X\leq 1$. Assume $X$ has exactly two extremal rays $R_1, R_2$ of length two.
If the contraction of $R_1$ is of type (E2), then $X\simeq \operatorname{Bl}_p(\mathbb{P}^3)$.
\end{proposition}

\begin{proof}
By Lemma~\ref{lem:noD2} and Proposition~\ref{prop:two:rays}, the contraction of $R_2$ is a $\mathbb{P}^1$-bundle $\psi: X\to S$. The exceptional divisor $E\simeq \mathbb{P}^2$ of the (E2) contraction $\varphi: X \to Y$ maps onto $S$ via $\psi$. By \cite{laz-im} (see also \cite[V.~Corollary~3.5]{KR}), $S\simeq \mathbb{P}^2$. Combining this with Proposition~\ref{prop:blowup-fano}, we see that $Y$ is a smooth Fano threefold with  $\rho(Y)=1$
satisfying
\[
4=\rho(X)+2=\tau_X\le \tau_Y=i(Y).
\]
Thus the Kobayashi--Ochiai theorem \cite{kooc} yields $Y\simeq \mathbb{P}^3$.
\end{proof}

\begin{proposition}
Let \(X\) be a smooth Fano threefold with \(c_X\le 1\).
If \(X\) has two extremal rays of length \(2\), both giving \(\mathbb P^1\)-bundle structures, then
\[
X \ \simeq\ \mathbb P^1\times\mathbb P^1\times\mathbb P^1,\qquad
\mathbb P(T_{\mathbb P^2}),\qquad
\text{or}\qquad
\mathbb P^1\times \operatorname{Bl}_p(\mathbb P^2).
\]
\end{proposition}

\begin{proof}
Let $f_i \colon X\to S_i$ ($i=1,2$) be the two $\mathbb P^1$-bundle structures associated with the two length-$2$ extremal rays.
By \cite[Thm.~2.2]{Kane17} (see also \cite[Thm.~2.3]{Kane22}), there exist smooth elementary contractions
$g_i:S_i\to Z$ fitting into a commutative diagram
\[
  \xymatrix{
    X \ar[r]^{f_1} \ar[d]_{f_2} & S_1 \ar[d]^{g_1} \\
    S_2 \ar[r]_{g_2} & Z.
  }
\]
Since each $f_i$ is a smooth contraction from a smooth Fano threefold,
\cite[Cor.~2.9]{KMM} implies that each $S_i$ is a smooth del Pezzo surface.
We distinguish cases according to $\dim Z$.

\medskip\noindent
\textbf{Case $\dim Z=0$.}
Then $S_1\simeq S_2\simeq \mathbb P^2$.
By the theorem of \cite[Thm.~1.2]{OSWW},
any smooth Fano threefold $X$ admitting two distinct $\mathbb{P}^1$-bundle
structures is isomorphic to a rational homogeneous variety of the form
$G/B$, where $G$ is a semisimple linear algebraic group and
$B\subset G$ a Borel subgroup.
Since $\rho(X)=2$, the group $G$ has rank two. For a semisimple linear
algebraic group $G$ with root system $\Phi$, one has
\[
  \dim(G/B)=|\Phi^{+}|,
\]
the number of positive roots. In rank two there are only the following
Dynkin types:
\[
  A_1\times A_1,\ A_2,\ B_2(=C_2),\ G_2,
\]
for which $|\Phi^{+}|$ equals $2,3,4,6$, respectively. Hence the only
possibility with $\dim(G/B)=3$ is the type $A_2$. Therefore
\[
  G/B \simeq \mathrm{SL}_3/B \simeq \mathbb{P}(T_{\mathbb{P}^2}),
\]
and we conclude $X\simeq \mathbb{P}(T_{\mathbb{P}^2})$.

\medskip\noindent
\textbf{Case $\dim Z=1$.}
Then $Z\simeq \mathbb P^1$, and each $g_i:S_i\to Z$ is a $\mathbb P^1$-bundle.
The only Hirzebruch surfaces that are del Pezzo surfaces are $\mathbb F_0\simeq \mathbb P^1\times\mathbb P^1$ and
$\mathbb F_1\simeq \operatorname{Bl}_p(\mathbb P^2)$, hence
\[
S_i\ \simeq\ \mathbb P^1\times\mathbb P^1\quad\text{or}\quad \operatorname{Bl}_p(\mathbb P^2).
\]
Consider the natural morphism
\[
\Psi \coloneqq (f_1,f_2):X\longrightarrow S_1\times_Z S_2.
\]
For each $z\in Z$, the fiber $X_z \coloneqq \Psi^{-1}(z)$ is smooth and admits two $\mathbb P^1$-bundle
structures over $g_i^{-1}(z)\simeq\mathbb P^1$.
A Hirzebruch surface $\mathbb F_e$ has two rulings \emph{iff} $e=0$; hence
$X_z\simeq \mathbb P^1\times\mathbb P^1$, and $\Psi|_{X_z}$ is an isomorphism.
Thus $\Psi$ is proper and quasi-finite of degree $1$, hence finite; the target is smooth
(since $g_i$ are smooth), thus normal, and Zariski’s Main Theorem yields
\begin{equation}\label{eq:fiberproduct}
X\ \simeq\ S_1\times_Z S_2.
\end{equation}

We analyze the possibilities for $(S_1,S_2)$.

\smallskip
\emph{(i) $S_1\simeq S_2\simeq \mathbb P^1\times\mathbb P^1$.}
We get
\[
X\ \overset{\eqref{eq:fiberproduct}}{\simeq}\ 
(\mathbb P^1\times\mathbb P^1)\times_{\mathbb P^1}(\mathbb P^1\times\mathbb P^1)
\ \simeq\ \mathbb P^1\times\mathbb P^1\times\mathbb P^1.
\]

\smallskip
\emph{(ii) $S_1\simeq \mathbb P^1\times\mathbb P^1$ and $S_2\simeq \operatorname{Bl}_p(\mathbb P^2)$.}
We get
\[
X\ \overset{\eqref{eq:fiberproduct}}{\simeq}\ 
(\mathbb P^1\times\mathbb P^1)\times_{\mathbb P^1}\operatorname{Bl}_p(\mathbb P^2)
\ \simeq\ \mathbb P^1\times \operatorname{Bl}_p(\mathbb P^2).
\]

\smallskip
\emph{(iii) $S_1\simeq S_2\simeq \operatorname{Bl}_p(\mathbb P^2)$.}
Write $S_i\simeq \mathbb F_1$ with negative section $C_i$ and fiber $F_i$,
and let $f_i:X\to S_i$ and $h:X\to Z$ be the projections in \eqref{eq:fiberproduct}.
For smooth fiber products over a curve,
\[
K_X\ =\ f_1^{\ast}K_{S_1/Z}\ +\ f_2^{\ast}K_{S_2/Z}\ +\ h^{\ast}K_Z,
\qquad
K_{S_i/Z}\ =\ K_{S_i}-g_i^{\ast}K_Z\ =\ -2C_i-F_i,
\]
while $K_Z=K_{\mathbb P^1}$ has degree $-2$.
Consider the curve $\Gamma \coloneqq C_1\times_Z C_2\simeq\mathbb P^1$ in $X$.
Then
\begin{eqnarray*}
K_X\cdot \Gamma &=& (K_{S_1/Z}\cdot C_1)+(K_{S_2/Z}\cdot C_2)+(K_Z\cdot h_{\ast}\Gamma) \\
&=& \bigl(-2C_1-F_1\bigr)\!\cdot C_1\ +\ \bigl(-2C_2-F_2\bigr)\!\cdot C_2\ -2 \\
&=& 1+1-2=0
\end{eqnarray*}
since on $\mathbb F_1$ one has $C_i^2=-1$ and $C_i\cdot F_i=1$.
Hence $(-K_X\cdot \Gamma)=0$, so $-K_X$ is not ample and $X$ is not Fano.
This case is excluded.

\medskip
Combining the $\dim Z=0$ conclusion with (i)–(ii) completes the proof.
\end{proof}

Summing up, we obtain the following:

\begin{theorem}\label{thm:smooth-fano-3fold-cx=1}
Let $X$ be a smooth Fano threefold with nef complexity $c_X\leq 1$. 
If $c_X<1$, then $c_X=0$ and hence $X$ is a product of projective spaces.
If $c_X=1$, then $X$ is isomorphic to one of
\[
Q^3,\quad \operatorname{Bl}_p(\mathbb{P}^3),\quad \mathbb{P}^1\times \operatorname{Bl}_p(\mathbb{P}^2),
\quad \mathbb{P}(T_{\mathbb{P}^2}),\quad 
\mathbb{P}\bigl(\mathcal{O}_{\mathbb{P}^1} (1)\oplus \mathcal{O}_{\mathbb{P}^1}^{\oplus 2}\bigr).
\]
\end{theorem}

\section{Fano varieties with many Mori fiber structures}\label{many MFS:section} In this section, we prove the Mukai conjecture for Fano varieties with many Mori fiber structures. Let us quickly explain the notation used here.
For a smooth projective variety $X$, let us consider the space of rational curves $\mathrm{RatCurves}^n(X)$, which admits a natural morphism $\mathrm{RatCurves}^n(X)\to \mathrm{Chow}(X)$  (for details, see \cite[Section~II.2]{KR}). A {\it family of rational curves} $\mathcal{M}$ on $X$ means an irreducible component of $\mathrm{RatCurves}^n(X)$.
The family $\mathcal{M}$ is called {\it unsplit} if it is proper. 
The family $\mathcal{M}$ is called a {\it dominating family} if the union of the rational curves parametrized by $\mathcal{M}$ is dense in $X$.
For a $K_X$-negative extremal contraction of fiber type 
$\varphi \colon  X \to Y$, we may find a dominating family of rational curves whose members are contracted by $\varphi$.
Among such families, 
one with the minimal anticanonical degree is called a {\it minimal dominating family of rational curves associated with $\varphi$}. 

By using the idea of proving \cite[Proposition~4~(4)]{mosw}, we show the following.

\begin{proposition}\label{linear-ind}Let $X$ be a smooth Fano variety, and let $R_1, \ldots, R_m \subset \overline{\mathrm{NE}}(X)$ be distinct extremal rays of fiber type. The rays $R_1, \ldots, R_m$ are linearly independent in $\mathrm{N}_1(X)$. In particular, $m\leq \rho(X)$.
\end{proposition}
 
\begin{proof} We denote by $\varphi_i \colon  X\to X_i$ the contraction of an extremal ray $R_i$, and we fix a curve $[C_i]\in R_i$ for any $i$. Let $\mathcal{M}_i$ be a
minimal dominating family of rational curves associated with $\varphi_i$. By a slight abuse of notation, we denote by $\overline{\mathcal{M}_i}$ the closure of the image of $\mathcal{M}_i\subset \mathrm{RatCurves}^n(X)$ in $\mathrm{Chow}(X)$. 

First, consider the case where $m\leq \rho(X)$. By induction on $m$, we prove that $R_1, \ldots, R_m$ are linearly independent. When $m=1$ or $m=2$, there is nothing to prove. So, assuming that $R_1, \ldots, R_{m-1}$ are linearly independent for $m$ $(3\leq m\leq \rho(X))$, we show that $R_1, \ldots, R_m$ are linearly independent. To prove this, assume the contrary, that is, $R_m\subset \langle [C_1], \ldots, [C_{m-1}]\rangle$. We have $[C_m]\equiv \sum_{i=1}^{m-1}a_i[C_i]$ for some $a_i \in \mathbb{Q}$. By the extremality of the ray $R_m$, there exists $j \in \{1, \ldots, m-1\}$ such that $a_j<0$. Without loss of generality, assume $a_{m-1}<0$.   
Consider the rc$(\overline{\mathcal{M}_1}, \ldots, \overline{\mathcal{M}_{m-2}})$-fibration $\pi \colon  X \dashrightarrow  Z$. Then \cite[Corollary~4.4]{aco} implies that $\dim Z>0$. Let $W\subset X$ be the indeterminacy locus of $\pi$.
Since $X$ is normal, the codimension of $W\subset X$ is at least $2$. Thus, we may consider $\pi^{\ast}H\in \mathrm{Pic}(X)$ for some ample divisor $H$ on $Z$. 
General fibers of $\pi$ are chains of curves from $\overline{\mathcal{M}_1}, \ldots, \overline{\mathcal{M}_{m-2}}$. 
Since $\varphi_{m-1}$ is a fiber type contraction, a general curve $C_{m-1}$ from $\mathcal{M}_{m-1}$ dominates $X$ 
and hence is not contracted by $\pi$. Therefore, $\pi^\ast H\cdot C_{m-1}>0$. We also have $\pi^{\ast}H\cdot C_i=0$ for $i=1,\ldots, m-2$.
 Thus, we obtain $\pi^{\ast}H\cdot C_m=a_{m-1}(\pi^{\ast}H\cdot C_{m-1})<0$. This is a contradiction. Consequently, $R_1, \ldots, R_m$ are linearly independent if $m\leq \rho(X)$.

Secondly, let us consider the case where $m> \rho(X)$. From the first part, $R_1, \ldots, R_{\rho(X)}$ are linearly independent. As in the first half, we have $[C_m]\equiv \sum_{i=1}^{\rho(X)}a_i'[C_i]$ for some $a_i' \in \mathbb{Q}$, and we may assume that $a_{\rho(X)}'<0$. Consider the rc$(\overline{\mathcal{M}_1}, \ldots, \overline{\mathcal{M}_{\rho(X)-1}})$-fibration $\pi' \colon  X \dashrightarrow Z'$ and apply the same argument as in the first half, we obtain a contradiction. Thus, we see that $m\leq \rho(X)$.
\end{proof}

The following theorem presents an interesting property of a smooth Fano variety such that the pseudo-effective cone $\mathrm{Pseff}(X)$ coincides with the nef cone $\mathrm{Nef}(X)$ in the N\'eron--Severi group.
This gives an answer to a question of J. Starr \cite{sta}.

\begin{theorem}\label{simpliciality}Let $X$ be a smooth Fano variety such that $\mathrm{Pseff}(X)=\mathrm{Nef}(X)$. Then the Kleiman--Mori cone $\overline{\mathrm{NE}}(X)$ and the nef cone $\mathrm{Nef}(X)$ are simplicial. 
\end{theorem}
 
\begin{proof} Since the dual of the cone $\overline{\mathrm{NE}}(X)$ is the cone $\mathrm{Nef}(X)$, it is enough to prove that $\overline{\mathrm{NE}}(X)$ is simplicial. By \cite[Lemma~4.4]{druel}, any elementary contraction of $X$ is of fiber type. Then $X$ admits $\rho(X)$-many elementary contractions of fiber type. Thus, our assertion follows from Proposition~\ref{linear-ind}. 
\end{proof}

\begin{theorem}\label{many MFS}Conjecture \ref{mukai conjecture}  holds for a smooth Fano variety $X$ with $\rho(X)$-many elementary contractions of fiber type. 
\end{theorem}

\begin{proof} Let $X$ be a smooth Fano variety with distinct extremal rays of fiber type $R_1, \ldots, R_{\rho(X)} \subset \overline{\mathrm{NE}}(X)$. We denote by $\varphi_j \colon  X\to X_j$ the contraction of an extremal ray $R_j$ and by $F_j$ a (smooth) general fiber of $\varphi_j$. Let $\mathcal{M}_j$ be a minimal dominating family of rational curves associated with $\varphi_j$.

To start, we prove the first part of Conjecture~\ref{mukai conjecture}.
By \cite[Theorem~2.2]{w-contr}, \cite{kooc}, and the divisibility $i(X)|i(F_j)$ which grants $i(F_j)-1 \geq i(X)-1$, it follows that
\begin{eqnarray}\label{M:1}
\dim X\geq \sum_{j=1}^{\rho(X)} \dim F_j \geq  \sum_{j=1}^{\rho(X)}(i(F_j)-1) \geq \rho(X)(i(X)-1).
\end{eqnarray}
Thus, the first part of Conjecture \ref{mukai conjecture} holds. Assume
$$
\dim X=\rho(X)(i(X)-1).
$$
It follows from (\ref{M:1}) that
$$
\dim X=\sum_{j=1}^{\rho(X)} \dim F_j =\sum_{j=1}^{\rho(X)}(i(F_j)-1)=\sum_{j=1}^{\rho(X)} (i(X)-1).
$$
Since $i(X) \leq i(F_j)$ for all $j$, we have $i(X)=i(F_j)$ for all $j$.
By \cite{kooc}, $\dim F_j \leq i(F_j)-1$ for all $j$.
Thus, we have $\dim F_j = i(F_j)-1$ for all $j$.
Furthermore,
again by \cite{kooc}, $F_j$ is isomorphic to a projective space $\mathbb{P}^{i(X)-1}$ and $\mathcal{M}_j$ contains a line in $\mathbb{P}^{i(X)-1}$. Since $-K_X=i(X)H$ for some ample divisor $H$, the length $\ell(R_j)$ is equal to $i(X)$. According to \cite[Lemma~2.7]{wat}, we see that $\mathcal{M}_j$ is unsplit.
Applying \cite[Theorem~1.1]{occ} and Proposition~\ref{linear-ind}, $X$ is isomorphic to a product of projective spaces. 
\end{proof}

\section{Smooth images of products of projective spaces}
\label{sec:smooth-images}

In this section, we prove the following:

\begin{theorem}\label{thm:smth-imgs}
Let $X\cong \prod_{i=1}^r\mathbb{P}^{n_i}$ be a product of projective spaces, and let $\phi\colon X\rightarrow Y$ be a surjective morphism to a smooth projective variety $Y.$
Then $Y\cong \prod_{j=1}^s\mathbb{P}^{m_j}$ is a product of projective spaces.
\end{theorem}

First, we prove a few lemmas.

\begin{lemma}\label{lem:reduce-to-fin-case}
Let $Y$ be a projective variety that receives a surjection from a product of projective spaces. Then $Y$ receives a finite surjection from a product of projective spaces.
\end{lemma}
\begin{proof}
    Let $X\cong \prod_{i=1}^r\mathbb{P}^{n_i}$ be a product of projective spaces, and let $f\colon X\rightarrow Y$ be a surjective morphism. Let $X\xrightarrow{g}W\xrightarrow{h}Y$ be the Stein factorization of $f.$ Since $X$ is a product of projective spaces, the contraction $g$ must be a projection $\prod_{i=1}^r\mathbb{P}^{n_i}\rightarrow \prod_{i\in I}\mathbb{P}^{n_i}$ for some $I\subset\{1,\hdots, r\}.$ Thus, the morphism $h$ is a finite surjection onto $Y$ from a product of projective spaces.
\end{proof}

\begin{lemma}\label{lem:shared-base}
Suppose given a commutative diagram 
\[
\xymatrix{ 
Z\ar[d]_-{\psi} & Z\times W\ar[l]_-{\pi_Z}\ar[d]_-{\phi} \ar[r]^-{\pi_W}&W \ar[d]^{{\rm Id}_W}\\ 
V &
Y\ar[l]_-{g} \ar[r]^-{f}&W
}
\]
satisfying the following conditions:
\begin{enumerate}
    \item $Y,$ $Z$ and $W$ are smooth projective varieties;
    \item $\psi$ and $\phi$ are finite and surjective; and
    \item $g$ and $f$ are contractions.
\end{enumerate}
Then $(g,f)\colon Y\xrightarrow{\cong}V\times W$ is an isomorphism. In particular, $V$ is smooth.
\end{lemma}
\begin{proof}
We begin by observing that it suffices to show that $g$ induces an isomorphism $f^{-1}(p)\xrightarrow{\cong}V$ for a general point $p\in W.$ Indeed, it will follow that $(g,f)\colon Y\rightarrow V\times W$ is a finite morphism of degree one between smooth projective varieties, hence an isomorphism. We note that $$\pi_Z|_{\pi_W^{-1}(p)}\colon \pi_W^{-1}(p)\rightarrow Z$$ is an isomorphism for all $p\in W$ and that $$\phi|_{\pi_W^{-1}(p)}\colon \pi_W^{-1}(p)\rightarrow f^{-1}(p)$$ is a finite and surjective morphism between smooth projective varieties for general $p\in W.$ Since we have \begin{align*}
    \deg(\psi)&=\deg((\psi\circ \pi_Z)|_{\pi_W^{-1}(p)})\\
    &=\deg((g\circ\phi)|_{\pi_W^{-1}(p)})\\
    &=\deg(g|_{f^{-1}(p)})\deg(\phi|_{\pi_W^{-1}(p)})
\end{align*}
for general $p\in W,$ it will suffice to show that $\deg(\phi|_{\pi_W^{-1}(p)})\geq \deg(\psi)$ for general $p\in W$ to obtain $\deg(g|_{f^{-1}(p)})=1$ and hence conclude.\par
Next, we show that $\deg(\psi)=\deg(\phi).$ Let $U_1\subset V$ be a nonempty open subset over which $\psi$
 is \'etale and $g$ is smooth, and let $U_2\subset Y$ be a nonempty open subset over which $\phi$ is \'etale. We note that $U_1\cap g(U_2)=g(g^{-1}(U_1)\cap U_2)$ is nonempty, and we fix $p \in U_1\cap g(U_2).$ Since $p\in U_1,$ $\psi^{-1}(p)$ consists of ${\rm deg}(\psi)$ reduced points.
 We note that $$\phi|_{\pi_Z^{-1}(q)}\colon \pi_Z^{-1}(q)\rightarrow g^{-1}(p)$$
is an isomorphism for each $q\in \psi^{-1}(p)$.
Indeed, $g^{-1}(p)$ is integral and of the same dimension as $\pi_Z^{-1}(q)$ and $\phi|_{\pi_Z^{-1}(q)}\colon \pi_Z^{-1}(q)\rightarrow Y$ is an embedding with left inverse given by $f.$ We note that $g^{-1}(p)\cap U_2$ is nonempty by our choice of $p$, and we fix $y\in g^{-1}(p)\cap U_2$.
On the one hand, the arguments above imply that the underlying reduced subscheme of $\phi^{-1}(y)$
consists of ${\rm deg}(\psi)$ points. On the other hand, $y\in U_2$ implies that $\phi^{-1}(y)$
consists of ${\rm deg}(\phi)$ reduced points. Thus, ${\rm deg}(\psi)={\rm deg}(\phi).$\par
Finally, we show that ${\rm deg}(\phi|_{\pi^{-1}_W(p)})\geq {\rm deg}(\phi)$ for general $p\in W.$ Let $U_3\subset W$ be a nonempty open subset over which $f$ is smooth. We note that $U_3\cap f(U_2)=f(f^{-1}(U_3)\cap U_2)$ is nonempty and that it is open in $U_3,$ hence in $W,$ since it is the image of $f^{-1}(U_3)\cap U_2$ under the smooth morphism $f|_{f^{-1}(U_3)}.$ Fix an arbitrary $p\in U_3\cap f(U_2).$ Then $$\phi|_{\pi_W^{-1}(p)}\colon \pi_W^{-1}(p)\rightarrow f^{-1}(p)$$ is a finite and surjective morphism between smooth projective varieties and $U_2\cap f^{-1}(p)\neq \emptyset.$ For $y\in U_2\cap f^{-1}(p),$ we see that $\phi^{-1}(y)$ consists of ${\rm deg}(\phi)$ reduced points and hence that the underlying reduced subscheme of $(\phi|_{\pi^{-1}_W(p)})^{-1}(y)$ consists of ${\rm deg}(\phi)$ points. Thus, ${\rm deg}(\phi|_{\pi^{-1}_W(p)})\geq {\rm deg}(\phi)$.
\end{proof}

\begin{lemma}\label{lem:img-fano-nef-is-pseff}
Let $\phi\colon X=\prod_{i=1}^r\mathbb{P}^{n_i}\rightarrow Y$ be a surjective morphism from a product of projective spaces onto a smooth projective variety $Y.$ Then $Y$ is a Fano manifold with ${\rm Pseff}(Y)={\rm Nef}(Y).$ Moreover, each nontrivial contraction of $Y$ is an equidimensional fibration. 
\end{lemma}
\begin{proof}
By Lemma~\ref{lem:reduce-to-fin-case}, we may assume that the morphism $\phi$ is finite.\par
We will begin by showing that each nontrivial contraction of $Y$ is an equidimensional fibration. Let $f\colon Y \rightarrow Z$ be a nontrivial contraction. Let $$X\xrightarrow{g}W\xrightarrow{\psi}Z$$ be the Stein factorization of $f\circ \phi$.
Thus, $g$ is a nontrivial contraction of $X$ and $\psi$ is finite and surjective.
Since $X$ is a product of projective spaces, it follows that $g$ is an equidimensional fibration. In particular, we have
$$
\dim Z = \dim W < \dim X = \dim 
Y,$$
hence that $f$ is a fibration. Write $m=\dim X -\dim W=\dim Y- \dim Z$ for the common relative dimension of $f$ and $g.$ To see that $f$ is equidimensional, let $p\in Z$ be a closed point. Since $\psi$ is a finite morphism, it follows that (the underlying reduced subscheme of) $\psi^{-1}(p)$ is a finite set of closed points in $W.$ Thus, each irreducible component of $g^{-1}\psi^{-1}(p)$ has dimension equal to $m.$ Since $g^{-1}\psi^{-1}(p)=\phi^{-1}f^{-1}(p)$ and $\phi$ is finite and surjective, it follows that each irreducible component of $f^{-1}(p)$ has dimension $m.$\par
Next, we show that $Y$ is Fano and that ${\rm Pseff}(Y)={\rm Nef}(Y).$ We begin by noting that $Y$ is of Fano type by ~\cite[Corollary 5.2]{FG12}. In particular, every nef divisor on $Y$ is semiample. Thus, the equality ${\rm Pseff}(Y)={\rm Nef}(Y)$ follows from the fact that every nontrivial contraction of $Y$ is a fibration as in the proof of ~\cite[Lemma 4.4]{Dru16}. Since $Y$ is of Fano type, there exists an effective $\mathbb Q$-divisor $\Delta$ on $Y$ such that $-(K_Y+\Delta)$ is ample. Since ${\rm Pseff}(Y)={\rm Nef}(Y),$ it follows that $\Delta$ is nef and hence that $-K_Y=-(K_Y+\Delta)+\Delta$ is ample. Thus, $Y$ is Fano.
\end{proof}

\begin{lemma}\label{lem:extr-contr-vert-divs}
Let $Y$ and $Z$ be locally factorial projective varieties of Fano type, and let $f\colon Y \rightarrow Z$ be an elementary contraction of fiber type. Then the prime divisors on $Y$ that are vertical over $Z$
are precisely those of the form $f^\ast D$ for a prime divisor $D$ on $Z.$
\end{lemma}
\begin{proof}
We use the local factoriality of $Z$ to make sense of $f^\ast D$ for an arbitrary prime divisor on $Z.$ Since $Y$ is locally factorial and $f$ is an elementary contraction from a Fano type variety, it follows from ~\cite[Theorem 3.7.(4)]{KM98} that $f^\ast D$ is reduced
and from ~\cite[Lemma 2.10]{Lai11} that $f^\ast D$ is irreducible for each prime divisor $D$ on $Z.$ Moreover, it follows from ~\cite[Lemma 2.10]{Lai11} that $f(E)$ is a divisor on $Z$ for each prime divisor $E$ on $Y$ that is vertical over $Z$.
The desired bijection follows.
\end{proof}

\begin{lemma}\label{lem:isotriv-over-smth-locus}
Consider a commutative diagram 
\[
\xymatrix{ 
X\ar[d]_-{g} \ar[r]^-{\phi}&Y \ar[d]^-{f}\\ 
W \ar[r]^-{\psi}&Z
}
\]
in which the following hold:
\begin{enumerate}
    \item $X$ is a product of projective spaces;
    \item $\phi$ is a finite surjective morphism onto a smooth projective variety $Y$;
    \item $f$ is an elementary contraction; and
    \item $\psi\circ g$ is the Stein factorization of $f\circ \phi.$
\end{enumerate}
Then the fibers of $f$ over $Z_{\rm reg}$ are isomorphic to projective space.
\end{lemma}
\begin{proof}
It follows from ~\cite[Corollary 5.2]{FG12} that both $Y$ and $Z$ are of Fano type. Since $X$ and $W$ are both products of projective spaces, it follows from ~\cite[Lemma 4.1]{Dru16} that ${\rm Nef}(X)={\rm Pseff}(X)$ and ${\rm Nef}(W)={\rm Pseff}(W)$. It then follows from ~\cite[Lemma 4.2]{Dru16} that the same holds for $Y$ and $Z.$ The contraction $f$ is a fibration by ~\cite[Lemma 4.4]{Dru16}, hence $Z$ is locally factorial by ~\cite[Corollary 4.8]{Dru16}.\par
Denote by $R_\phi$ and $R_\psi$ the ramification divisors of $\phi$ and $\psi,$ respectively. We claim that $R_\phi\geq g^\ast R_\psi.$ Granting this claim, we proceed as follows. We note that the relative anticanonical divisor $-K_{Y/Z}$ satisfies

$$
\phi^\ast (-K_{Y/Z})=-K_{X/W}+(R_\phi-g^\ast R_\psi).
$$

The relative anticanonical divisor $-K_{X/W}$ is nef, and $R_\phi-g^\ast R_\psi$ is pseudo-effective by the claim. But ${\rm Nef}(X)={\rm Pseff}(X),$ so we see that $\phi^\ast (-K_{Y/Z})$ is nef. It follows that $-K_{Y/Z}$ is nef, and hence from (the proof of) ~\cite[Theorem A.13]{CPZ19} that $f$ is isotrivial over $Z_{reg}.$ As general fibers of $f$ are smooth by generic smoothness, this implies that $f$ is a smooth morphism over $Z_{reg}.$ Since $Z_{\rm reg}$ is simply connected by ~\cite[Lemma 3.13]{Dru16}, it follows from ~\cite[Theorem 4.18]{Voi03b} that the general fibers $F$ of the elementary contraction $f$ have $\rho(F)=1$. The desired conclusion now follows from ~\cite[Theorem 1]{OW02}.\par
We turn now to prove the claim. We note that $g^\ast D$ is a prime divisor on $X$ for all prime divisors $D$ on $W.$ Thus, we must show that 
$$
{\rm coeff}_{D}(R_\psi)\leq {\rm coeff}_{g^\ast D}(R_\phi)
$$
holds for all prime divisors $D$ on $W$.
For the rest of the proof, given a finite morphsim $h\colon U\rightarrow V$ and a prime divisor $E$ on $U$, we will denote by $h(E)$ the prime divisor on $V$ with support equal to that of the image of $E.$ Localizing at the generic points of $D$ and $\psi(D)$ and appealing to the Riemann--Hurwitz formula, we see that
$$
{\rm coeff}_{D}(R_\psi)={\rm coeff}_D(\psi^\ast \psi(D))-1.
$$
Similarly, we see that
$$
{\rm coeff}_{g^\ast D}(R_\phi)={\rm coeff}_{g^\ast D}(\phi^*\phi(g^*D))-1.
$$
We note that $\phi(g^*D)=f^*\psi(D)$. Indeed, $g^*D$ and $f^*\psi(D)$ are both prime divisors by Lemma~\ref{lem:extr-contr-vert-divs} and $\phi$ is finite, so it suffices to observe that $g^*D\leq g^*\psi^*\psi(D)=\phi^*f^*\psi(D).$ It follows that $\phi^*\phi(g^*D)=g^*\psi^*\psi(D),$ and hence that $${\rm coeff}_{g^*D}(R_\phi)={\rm coeff}_{g^*D}(g^*\psi^*\psi(D))-1.$$ It follows from the smoothness of $g$ that $${\rm coeff}_D(\psi^*\psi(D))={\rm coeff}_{g^*D}(g^*\psi^*\psi(D)),$$ so we win. 
\end{proof}

\begin{notation}\label{not:complementary-faces}
Let $V$ be a finite-dimensional real vector space and let $C\subset V$ be a simplicial cone.
Thus, there are linearly independent vectors $v_1,\hdots, v_n\in C$ such that $C={\rm Cone}(v_1,\hdots v_n)$ is the convex cone generated by these vectors. Given a subset $I\subset \{1,\hdots, r\},$ we will say that the faces ${\rm Cone}(v_i\vert i\in I)$ and ${\rm Cone}(v_i\vert i\notin I)$ are \textit{complementary faces} of $C$.
\end{notation}

\begin{proof}[Proof of Theorem~\ref{thm:smth-imgs}]
By Lemma~\ref{lem:reduce-to-fin-case}, we may assume that the morphism $\phi$ is finite. Replacing $\phi$ if necessary, we will assume moreover that $\phi$ is of minimal degree among all finte surjections from a product of projective spaces onto $Y.$ \par
We claim that $\phi$ is an isomorphism. Since $Y$ is a Fano manifold by Lemma~\ref{lem:img-fano-nef-is-pseff}, to show that $\phi$ is an isomorphism it suffices to show that $\phi$ is \'etale. Since $Y$ is smooth, the branch locus of $\phi$ is purely divisorial. Writing $B$ for the reduced sum of the irreducible components of this divisor, to show that $\phi$ is \'etale it will suffice to show that $B=0.$ Since $B$ is an effective Cartier divisor on the projective variety $Y$, to show that $B=0$ it will suffice to show that $B$ is numerivally trivial. Since every elementary contraction of $Y$ is a fibration by Lemma~\ref{lem:img-fano-nef-is-pseff}, to show that $B$ is numerically trivial it will suffice to show that it is vertical for every elementary contraction of $Y.$\par
Let $f\colon Y\rightarrow Z$ be an elementary contraction, and let $X\xrightarrow{g}W\xrightarrow{\psi}Y$ be the Stein factorization of $f\circ \phi.$ Denote by $\widetilde{X}$ the normalization of the main component of $W\times_{Z}Y.$ We obtain a commutative diagram 
\[
\xymatrix{
X \ar@/_/[ddr]_{g} \ar@/^/[drr]^{\phi} \ar@{->}[dr]|-{h}\\
&\widetilde{X} \ar[d]^{\widetilde{g}} \ar[r]^{\widetilde{\phi}} & Y\ar[d]^{f} \\
& W\ar[r]^{\psi} &Z.
}
\]
Note that $h$ and $\widetilde{\phi}$ must both be finite and surjective. Since $\widetilde{X}$ is normal and is the image of the Fano manifold $X,$ it follows from ~\cite[Theorem 5.1]{FG12} that $\widetilde{X}$ is of Fano type. In particular, $\widetilde{X}$ is Cohen--Macaulay. The morphism $\widetilde{g}$ is equidimensional since $g$ is equidimensional and $h$ is finite and surjective.
It follows from miracle flatness that $\widetilde{g}$ is flat.
The fibers of $f$ over $Z_{\rm reg}$ are all isomorphic to projective space by Lemma~\ref{lem:isotriv-over-smth-locus}, hence the same is true of the fibers of $\widetilde{g}$ over $\psi^{-1}(Z_{\rm reg}).$ Since $\psi$ is finite, it follows that the complement of $\psi^{-1}(Z_{\rm reg})$ has codimension at least two in $W.$ It now follows from ~\cite[Theorem 3]{ARM14} that $\widetilde{g}$ is a smooth morphism. In particular, $\widetilde{X}$ is smooth. \par
We claim that $h$ is an isomorphism. To show this, it will suffice by our choice of $\phi$ to show that $\widetilde{X}$ is a product of projective spaces. Since $\overline{\mathrm{NE}}(\widetilde{X})$ is simplicial by Lemma~\ref{lem:img-fano-nef-is-pseff} and Theorem ~\ref{simpliciality}, we may consider the contraction $\widehat{g}\colon \widetilde{X}\rightarrow \widehat{W}$ of the face of $\overline{\mathrm{NE}}(\widetilde{X})$ complementary in the sense of Notation ~\ref{not:complementary-faces} to the face contracted by $\widetilde{g}$. Let $X\xrightarrow{\widehat{h}} \widehat{X}\xrightarrow{\pi}\widehat{W}$ be Stein factorization of $\widehat{g}\circ h.$ The diagram 
\[
\xymatrix{ 
\widehat{X}\ar[d]_-{\pi} & X\ar[l]_-{\widehat{h}}\ar[d]_-{h} \ar[r]^-{g}&W \ar[d]^{{\rm Id}_W}\\ 
\widehat{W} &
\widetilde{X}\ar[l]_-{\widehat{g}} \ar[r]^-{\widetilde{g}}&W
}
\]
satisfies the hypotheses of Lemma~\ref{lem:shared-base}. It follows that $\widetilde{X}\cong W\times \widehat{W}$ and that $\widetilde{g}$ is the projection away from $\widehat{W}.$ Thus, $\widehat{W}$ must be isomorphic to projective space, and $\widetilde{X}$ must be isomorphic to a product of projective spaces. Thus,  by the minimality of $\deg \phi$, we conclude that $h$ must be an isomorphism. In particular, it follows that $X$ is the normalization of the main component of $Y \times_Z W$.
\par
Since $W$ is smooth, there is a nonempty open $U\subset Z$ over which the finite morphism $\psi$ is \'etale. Since $X$ is the normalization of the main component of $Y \times_Z W$, it follows that $\phi$ is \'etale over $f^{-1}(U).$ Thus, $B$ must have support in the complement of $f^{-1}(U)$ and must be vertical over $Z$.
\end{proof}

\bibliographystyle{habbvr}
\bibliography{bib}

\end{document}